\documentclass[11pt]{article}%
\usepackage{amsfonts}
\usepackage{amsmath}
\usepackage{scalefnt}
\usepackage{amssymb}
\usepackage{graphicx}
\usepackage{geometry}%
\newtheorem{theorem}{Theorem}

\newtheorem{conjecture}[theorem]{Conjecture}
\newtheorem{corollary}[theorem]{Corollary}

\newtheorem{definition}[theorem]{Definition}

\newtheorem{problem}[theorem]{Problem}
\newtheorem{proposition}[theorem]{Proposition}

\newenvironment{proof}[1][Proof]{\noindent\textbf{#1.} }{\ \rule{0.5em}{0.5em}}
\begin{document}

\title{Rational Orthogonal \emph{versus} Real Orthogonal}
\author{Dragomir \v{Z}. Djokovi\'{c}\thanks{Department of Pure Mathematics, University
of Waterloo, Waterloo N2L 3G1, ON Canada; djokovic@uwaterloo.ca}
\and Simone Severini\thanks{Institute for Quantum Computing \ and Department of
Combinatorics \& Optimization, \ University of Waterloo, \ Waterloo N2L 3G1,
ON Canada; simoseve@gmail.com}
\and Ferenc Sz\"{o}ll\H{o}si\thanks{Institute of Mathematics and its Applications,
Central European University (CEU), H-1051 Budapest, N\'ador u. 9, Hungary;
szoferi@gmail.com}}
\maketitle

\begin{abstract}
The main question we raise here is the following one: given a real orthogonal
$n\times n$ matrix $X$, is it true that there exists a rational orthogonal
matrix $Y$ having the same zero-pattern? We conjecture that this is the case
and prove it for $n\leq5$. We also consider the related problem for symmetric
orthogonal matrices.

\medskip

\textbf{Keywords and phrases:} Real and rational orthogonal matrices,
zero-patterns, combinatorial orthogonality.

\end{abstract}

\section{Introduction}

Let $X=[X_{i,j}]$ be an $m\times n$ matrix over any field. The
\emph{zero-pattern }of $X$, denoted by $\underline{X}=[\underline{X}_{i,j}]$,
is the $m\times n$ $(0,1)$-matrix such that
\[
\underline{X}_{i,j}=\left\{
\begin{tabular}
[c]{ll}%
$1,$ & if $X_{i,j}\neq0;$\\
$0,$ & if $X_{i,j}=0.$%
\end{tabular}
\ \right.
\]
We shall say that $\underline{X}$ is the \emph{support} of $X$.

A square matrix $X$ is said to be \emph{unitary} if its entries are complex
and $XX^{\dagger}=I$, where $X^{\dagger}$ is the transpose conjugate of $X$
and $I$ is the identity matrix. A square matrix $X$ is said to be \emph{real
orthogonal} (or, equivalently, \emph{orthogonal}) if its entries are real and
$XX^{T}=I$, where $X^{T}$ is the transpose of $X$. A square matrix $X$ is said
to be \emph{rational orthogonal} if it is orthogonal and its entries are
rational. The sets of unitary, orthogonal, and rational orthogonal matrices of
size $n$ are denoted by $U(n)$, $O(n)$ and $O_{n}(\mathbb{Q})$, respectively.

The notion and the study of the zero-patterns of unitary matrices go back to
\cite{F88} (see also \cite{F63}) in the mathematical context, and to
\cite{lou} (see also \cite{lan}), motivated by foundational questions in
quantum mechanics. An extended list of references on this topic is contained
in \cite{sz}. For a comprehensive reference in matrix theory see, \emph{e.g.},
\cite{hj}.

When discussing properties of zero-patterns, it is natural to ask whether the
number field influences their structure. Specifically, in this paper we
formulate and support the following two conjectures.

\begin{conjecture}
\label{con1}For any $X\in U(n)$ there exists $Y\in O(n)$ such that
$\underline{X}=\underline{Y}$.
\end{conjecture}

\begin{conjecture}
\label{con2}For any $Y\in O(n)$ there exists $Z\in O_{n}(\mathbb{Q})$ such
that $\underline{Y}=\underline{Z}$.
\end{conjecture}

Our main tool of analysis will be the notion of a strongly quadrangular matrix
introduced in \cite{sev1}. This extends naturally the concept of
\emph{quadrangularity} (or, equivalently, \emph{combinatorial orthogonality}
\cite{bb}). A matrix $X$ is said to be \emph{quadrangular} if every two rows
and every two columns \textquotedblleft intersect\textquotedblright\ in more
than a single entry whenever their intersection is nonempty. In other words,
the inner product of every two rows and every two columns of $\underline{X}$
is not $1$. Let $X=[X_{i,j}]$ be a complex $m\times n$ matrix. We write $X>0$
if all $X_{i,j}>0$. For $R\subseteq\{1,2,...,m\}$ and $C\subseteq
\{1,2,...,n\}$, we denote by $X_{R}^{C}$ the $|R|\times|C|$ submatrix of $X$
in the intersection of the rows and the columns indexed by $R$ and $C$, respectively.

\begin{definition}
[Strongly quadrangular matrix]We say that an $m\times n$ $\{0,1\}$-matrix
$X=[X_{i,j}]$ is \emph{row strongly quadrangular} (\emph{RSQ}) if there does
not exist $R\subseteq\{1,2,...,m\}$ with $\left\vert R\right\vert \geq2$ such
that, defining $R^{\prime}=\{k:X_{i,k}X_{j,k}=1$ for some $i\neq j$ in $R\}$,
we have $\left\vert R^{\prime}\right\vert <\left\vert R\right\vert $ and
$X_{R}^{R^{\prime}}$ has no zero-rows. We say that an $m\times n$
$\{0,1\}$-matrix $X$ is \emph{strongly quadrangular} (\emph{SQ}) if both $X$
and $X^{T}$ are RSQ.
\end{definition}

In \cite{sev1}, it was proved that if $X\in U(n)$ then $\underline{X}$ is SQ,
but the converse is not necessarily true (see also \cite{lu}). Proposition
\ref{141516} below gives the smallest possible SQ zero-patterns that do not
support unitary matrices.

So far we have been unable to exhibit counterexamples which would disprove
Conjecture \ref{con1} or Conjecture \ref{con2}. We can however get a feeling
about the problem, by explicitly working out concrete situations. For
instance, Beasley, Brualdi and Shader \cite{bb} have shown that if $X$ is a
real matrix with zero-pattern%
%TCIMACRO{\TeXButton{\begingroup}{\begingroup}}%
%BeginExpansion
\begingroup
%EndExpansion%
%TCIMACRO{\TeXButton{\scalefont{x}}{\scalefont{1}} }%
%BeginExpansion
\scalefont{1}
%EndExpansion
\[
\underline{X}=\left[
\begin{array}
[c]{ccccccccccc}%
1 & 1 & 0 & 1 & 0 & 0 & 1 & 0 & 0 & 0 & 1\\
1 & 1 & 1 & 0 & 1 & 0 & 0 & 1 & 0 & 0 & 0\\
0 & 1 & 1 & 1 & 0 & 1 & 0 & 0 & 1 & 0 & 0\\
0 & 0 & 1 & 1 & 1 & 0 & 1 & 0 & 0 & 1 & 0\\
0 & 0 & 0 & 1 & 1 & 1 & 0 & 1 & 0 & 0 & 1\\
1 & 0 & 0 & 0 & 1 & 1 & 1 & 0 & 1 & 0 & 0\\
0 & 1 & 0 & 0 & 0 & 1 & 1 & 1 & 0 & 1 & 0\\
0 & 0 & 1 & 0 & 0 & 0 & 1 & 1 & 1 & 0 & 1\\
1 & 0 & 0 & 1 & 0 & 0 & 0 & 1 & 1 & 1 & 0\\
0 & 1 & 0 & 0 & 1 & 0 & 0 & 0 & 1 & 1 & 1\\
1 & 0 & 1 & 0 & 0 & 1 & 0 & 0 & 0 & 1 & 1
\end{array}
\right]
\]%
%TCIMACRO{\TeXButton{\endgroup}{\endgroup}}%
%BeginExpansion
\endgroup
%EndExpansion
then $X\notin O(11)$. Once verified that that $\underline{X}$ is SQ, we
observe that $X$ is not a candidate for a counterexample to Conjecture
\ref{con1}, therefore corroborating the idea that the number field does not
have a strong role in determining a zero-pattern. We can proceed as follows in
four steps:

\begin{itemize}
\item By multiplying the columns $1,2,4,7,$ and $11$ by suitable phase
factors, all entries in the first row are real;

\item By multiplying the rows $2,3,5,6,7,9,10$ and $11$ by phase factors, the
entries
\[
(2,1),(3,2),(5,4),(6,1),(7,2),(9,1),(10,2)\text{ and }(11,1)
\]
are real;

\item By multiplying the columns $3,5,6,8$ and $9$ by phase factors, the
entries
\[
(2,3),(2,5),(2,8),(3,6)\text{ and }(3,9)
\]
are real;

\item Finally, by multiplying the rows $4$ and $8$, and the column $10$ by
phase factors, the entries
\[
(4,3),(8,3)\text{ and }(4,10)
\]
are real.
\end{itemize}

At this point, all the entries mentioned above are real. If $X\in U(n)$ then
the inner products of different rows of the matrix obtained with these steps
must vanish. It follows that $X$ is a real matrix, but we know that $X\notin
O(11)$ by \cite{bb}.

Here we adopt a systematic approach to our conjectures. In Section \ref{sec2},
we verify Conjecture \ref{con1} and Conjecture \ref{con2} for all
$(0,1)$-matrices of size $n\leq5$. For this purpose, we use the tables in
\cite{sz} of all SQ $(0,1)$-matrices of small size. On the way, we prove that
some of those are not zero-patterns of unitary matrices, thus refining the
classification of \cite{sz}. In Section \ref{sec3}, we construct examples of
symmetric rational orthogonal matrices with specified indecomposable
zero-pattern and specified trace. In Section \ref{sec4}, we construct some
infinite families of rational orthogonal matrices. The constructions are based
on orthogonal designs, graphs and combinatorial arguments. We conclude our
paper in Section \ref{sec5} with three intriguing open problems.

Recall that an $n\times n$ matrix $X$ that contains an $s\times(n-s)$ zero
submatrix for $0<s<n$ is said to be \emph{decomposable}. If no such submatrix
exists then $X$ is said to be \emph{indecomposable}.

\section{Rational orthogonal matrices of small size\label{sec2}}

We shall consider the indecomposable SQ zero-patterns of size $n\leq5$. Two
$(0,1)$-matrices $X$ and $Y$ are said to be \emph{equivalent }if there are
permutation matrices $P$ and $Q$ such that $PXQ=Y$. A list of representatives
for equivalence classes of indecomposable SQ zero-patterns of size $n\leq5$
was drawn in \cite{sz}. We construct rational orthogonal matrices for each
specific zero-pattern. This is not possible for the cases $14,15$ and $16$,
because it turns out that those do not support unitary matrices. Here is a
formal statement of such a fact:

\begin{proposition}
\label{141516}There is no matrix $X\in U(5)$ such that $\underline{X}$ is one
of the following zero-patterns:%
\[%
\begin{tabular}
[c]{lll}%
$\left[
\begin{array}
[c]{ccccc}%
0 & 0 & 1 & 1 & 1\\
0 & 0 & 1 & 1 & 1\\
1 & 1 & 0 & 1 & 1\\
1 & 1 & 1 & 1 & 1\\
1 & 1 & 1 & 1 & 1
\end{array}
\right]  _{14},$ & $\left[
\begin{array}
[c]{ccccc}%
0 & 0 & 1 & 1 & 1\\
0 & 0 & 1 & 1 & 1\\
1 & 1 & 0 & 1 & 1\\
1 & 1 & 1 & 0 & 1\\
1 & 1 & 1 & 1 & 1
\end{array}
\right]  _{15},$ & $\left[
\begin{array}
[c]{ccccc}%
0 & 0 & 1 & 1 & 1\\
0 & 0 & 1 & 1 & 1\\
1 & 1 & 0 & 1 & 1\\
1 & 1 & 1 & 0 & 1\\
1 & 1 & 1 & 1 & 0
\end{array}
\right]  _{16}$%
\end{tabular}
.
\]

\end{proposition}

\begin{proof}
Suppose that such $X$ exists. Let $Y=X_{\{3,4,5\}}^{\{1,2\}}$ and
$Z=X_{\{3,4,5\}}^{\{3,4\}}$. By inspecting the above three zero-patterns, we
conclude that $Z$ has rank $2$. Since $X$ is unitary, we have $Y^{\dagger}Z=0$
and so the two columns of $Y$ must be linearly dependent. Consequently, the
first two columns of $X$ are linearly dependent, which is a contradiction.
\end{proof}

\bigskip

The main result of the paper is essentially the following theorem:

\begin{theorem}
Conjecture \ref{con1} and Conjecture \ref{con2} are true for $n\leq5$.
\end{theorem}

\begin{proof}
Clearly it suffices to consider only the indecomposable SQ zero-patterns. For
$n\leq5$, these zero-patterns have been enumerated in \cite{sz} (up to
equivalence). All of these zero-patterns support unitary matrices, except the
three cases for $n=5$ mentioned in Proposition \ref{141516}. Thus, in order to
prove the theorem it suffices to construct a matrix in $O_{n}(\mathbb{Q})$ for
each of the remaining zero-patterns. This is done in the list below. Many of
these matrices have been constructed by using an exhaustive search but in some
cases this was not possible and we have resorted to \emph{ad hoc} methods.
Some of these methods are sketched in Section \ref{sec4}.
\end{proof}

\bigskip

In our list, a matrix $X$ will be written in the form%
\[
X=\frac{1}{d}\left[
\begin{array}
[c]{ccc}%
\ast & \cdots & \ast\\
\vdots & \ddots & \vdots\\
\ast & \cdots & \ast
\end{array}
\right]  _{k},
\]
where $k$ is simply a numerical label for the equivalence classes of
zero-patterns, identical to the labels in \cite{sz}. We say that the
denominator $d$ is \emph{minimal} if it is the smallest possible denominator
among all rational orthogonal matrices with the same zero-pattern. All
denominators in the list are minimal except for a few cases, when $n=5$.
Exceptions are the cases 6, 7, 19, 28, and 31.

\subsection{$n=2$}%

%TCIMACRO{\TeXButton{\begingroup}{\begingroup}}%
%BeginExpansion
\begingroup
%EndExpansion%
%TCIMACRO{\TeXButton{\scalefont{x}}{\scalefont{1}}}%
%BeginExpansion
\scalefont{1}%
%EndExpansion%
\[%
\begin{tabular}
[c]{l}%
$\frac{1}{5}\left[
\begin{array}
[c]{rr}%
3 & 4\\
4 & -3
\end{array}
\right]  _{1}$%
\end{tabular}
.
\]
%

%TCIMACRO{\TeXButton{\endgroup}{\endgroup}}%
%BeginExpansion
\endgroup
%EndExpansion

\subsection{$n=3$}%

%TCIMACRO{\TeXButton{\begingroup}{\begingroup}}%
%BeginExpansion
\begingroup
%EndExpansion%
%TCIMACRO{\TeXButton{\scalefont{x}}{\scalefont{1}}}%
%BeginExpansion
\scalefont{1}%
%EndExpansion%
\[%
\begin{tabular}
[c]{lll}%
$\frac{1}{25}\left[
\begin{array}
[c]{rrr}%
16 & 12 & 15\\
12 & 9 & -20\\
15 & -20 & \mathbf{0}%
\end{array}
\right]  _{1}$ &  & $\frac{1}{3}\left[
\begin{array}
[c]{rrr}%
2 & -1 & 2\\
-1 & 2 & 2\\
2 & 2 & -1
\end{array}
\right]  _{2}$%
\end{tabular}
.
\]
%

%TCIMACRO{\TeXButton{\endgroup}{\endgroup}}%
%BeginExpansion
\endgroup
%EndExpansion

\subsection{$n=4$}%

%TCIMACRO{\TeXButton{\begingroup}{\begingroup}}%
%BeginExpansion
\begingroup
%EndExpansion%
%TCIMACRO{\TeXButton{\scalefont{x}}{\scalefont{1}}}%
%BeginExpansion
\scalefont{1}%
%EndExpansion%
\[%
\begin{tabular}
[c]{ccc}%
$\frac{1}{9}\left[
\begin{array}
[c]{rrrr}%
8 & -3 & 2 & 2\\
-3 & \mathbf{0} & 6 & 6\\
2 & 6 & -4 & 5\\
2 & 6 & 5 & -4
\end{array}
\right]  _{1}$ &  & $\frac{1}{9}\left[
\begin{array}
[c]{rrrr}%
6 & \mathbf{0} & 3 & -6\\
\mathbf{0} & 1 & 8 & 4\\
3 & 8 & -2 & 2\\
-6 & 4 & 2 & -5
\end{array}
\right]  _{2}$%
\end{tabular}
\
\]%
\[%
\begin{tabular}
[c]{ccc}%
$\frac{1}{33}\left[
\begin{array}
[c]{rrrr}%
16 & 7 & \mathbf{0} & -28\\
7 & \mathbf{0} & 32 & 4\\
\mathbf{0} & 32 & -1 & 8\\
-28 & 4 & 8 & -15
\end{array}
\right]  _{3}$ &  & $\frac{1}{3}\left[
\begin{array}
[c]{rrrr}%
1 & 2 & -2 & \mathbf{0}\\
2 & \mathbf{0} & 1 & 2\\
-2 & 1 & \mathbf{0} & 2\\
\mathbf{0} & 2 & 2 & -1
\end{array}
\right]  _{4}$%
\end{tabular}
\
\]%
\[%
\begin{tabular}
[c]{lll}%
$\frac{1}{65}\left[
\begin{array}
[c]{rrrr}%
25 & \mathbf{0} & -36 & 48\\
\mathbf{0} & \mathbf{0} & 52 & 39\\
-36 & 52 & -9 & 12\\
48 & 39 & 12 & -16
\end{array}
\right]  _{5}$ &  & $\frac{1}{15}\left[
\begin{array}
[c]{rrrr}%
\mathbf{0} & \mathbf{0} & 9 & 12\\
2 & 14 & -4 & 3\\
10 & -5 & -8 & 6\\
11 & 2 & 8 & -6
\end{array}
\right]  _{6}$%
\end{tabular}
\
\]%
\[%
\begin{tabular}
[c]{lll}%
$\frac{1}{25}\left[
\begin{array}
[c]{rrrr}%
15 & \mathbf{0} & 12 & 16\\
\mathbf{0} & 20 & 12 & -9\\
-20 & \mathbf{0} & 9 & 12\\
\mathbf{0} & 15 & -16 & 12
\end{array}
\right]  _{7}$ &  & $\frac{1}{2}\left[
\begin{array}
[c]{rrrr}%
1 & 1 & 1 & -1\\
1 & 1 & -1 & 1\\
1 & -1 & 1 & 1\\
-1 & 1 & 1 & 1
\end{array}
\right]  _{8}$%
\end{tabular}
\ .
\]
%

%TCIMACRO{\TeXButton{\endgroup}{\endgroup}}%
%BeginExpansion
\endgroup
%EndExpansion

\subsection{$n=5$}%

%TCIMACRO{\TeXButton{\begingroup}{\begingroup}}%
%BeginExpansion
\begingroup
%EndExpansion%
%TCIMACRO{\TeXButton{\scalefont{x}}{\scalefont{1}}}%
%BeginExpansion
\scalefont{1}%
%EndExpansion%
\[%
\begin{tabular}
[c]{lll}%
$\frac{1}{4}\left[
\begin{array}
[c]{rrrrr}%
3 & 1 & 1 & 1 & -2\\
1 & 3 & -1 & -1 & 2\\
1 & -1 & 3 & -1 & 2\\
1 & -1 & -1 & 3 & 2\\
-2 & 2 & 2 & 2 & \mathbf{0}%
\end{array}
\right]  _{1}$ &  & $\frac{1}{7}\left[
\begin{array}
[c]{rrrrr}%
4 & -3 & 2 & 4 & 2\\
-3 & 4 & 2 & 4 & 2\\
2 & 2 & 3 & -4 & 4\\
4 & 4 & -4 & 1 & \mathbf{0}\\
2 & 2 & 4 & \mathbf{0} & -5
\end{array}
\right]  _{2}$%
\end{tabular}
\
\]%
\[%
\begin{tabular}
[c]{lll}%
$\frac{1}{11}\left[
\begin{array}
[c]{rrrrr}%
8 & 4 & 1 & 2 & -6\\
4 & 5 & -4 & \mathbf{0} & 8\\
1 & -4 & \mathbf{0} & 10 & 2\\
2 & \mathbf{0} & 10 & -1 & 4\\
-6 & 8 & 2 & 4 & -1
\end{array}
\right]  _{3}$ &  & $\frac{1}{5}\left[
\begin{array}
[c]{rrrrr}%
3 & 2 & -2 & 2 & 2\\
2 & 2 & 1 & -4 & \mathbf{0}\\
-2 & 1 & 2 & \mathbf{0} & 4\\
2 & -4 & \mathbf{0} & -1 & 2\\
2 & \mathbf{0} & 4 & 2 & -1
\end{array}
\right]  _{4}$%
\end{tabular}
\
\]%
\[%
\begin{tabular}
[c]{lll}%
$\frac{1}{147}\left[
\begin{array}
[c]{rrrrr}%
145 & 8 & \mathbf{0} & 14 & -18\\
8 & 51 & 80 & -112 & \mathbf{0}\\
\mathbf{0} & 80 & 47 & 70 & 90\\
14 & -112 & 70 & \mathbf{0} & 63\\
-18 & \mathbf{0} & 90 & 63 & -96
\end{array}
\right]  _{5}$ &  & $\frac{1}{625}\left[
\begin{array}
[c]{rrrrr}%
256 & 240 & 192 & 375 & 300\\
240 & 225 & 180 & \mathbf{0} & -500\\
192 & 180 & 144 & -500 & 225\\
375 & \mathbf{0} & -500 & \mathbf{0} & \mathbf{0}\\
300 & -500 & 225 & \mathbf{0} & \mathbf{0}%
\end{array}
\right]  _{6}$%
\end{tabular}
\
\]%
\[%
\begin{tabular}
[c]{lll}%
$\frac{1}{75}\left[
\begin{array}
[c]{rrrrr}%
50 & -25 & \mathbf{0} & 30 & 40\\
-25 & 50 & \mathbf{0} & 30 & 40\\
\mathbf{0} & \mathbf{0} & \mathbf{0} & 60 & -45\\
30 & 30 & 60 & -9 & -12\\
40 & 40 & -45 & -12 & -16
\end{array}
\right]  _{7}$ &  & $\frac{1}{25}\left[
\begin{array}
[c]{rrrrr}%
16 & 12 & \mathbf{0} & 12 & -9\\
12 & 9 & \mathbf{0} & -16 & 12\\
\mathbf{0} & \mathbf{0} & \mathbf{0} & 15 & 20\\
12 & -16 & 15 & \mathbf{0} & \mathbf{0}\\
-9 & 12 & 20 & \mathbf{0} & \mathbf{0}%
\end{array}
\right]  _{8}$%
\end{tabular}
\
\]%
\[%
\begin{tabular}
[c]{lll}%
$\frac{1}{9}\left[
\begin{array}
[c]{rrrrr}%
8 & 2 & 2 & \mathbf{0} & -3\\
2 & 4 & -6 & 3 & 4\\
2 & -6 & 1 & 6 & 2\\
\mathbf{0} & 3 & 6 & \mathbf{0} & 6\\
-3 & 4 & 2 & 6 & -4
\end{array}
\right]  _{9}$ &  & $\frac{1}{9}\left[
\begin{array}
[c]{rrrrr}%
\mathbf{0} & \mathbf{0} & 6 & 3 & 6\\
\mathbf{0} & 4 & 5 & 2 & -6\\
6 & -5 & \mathbf{0} & 4 & -2\\
3 & 6 & -4 & 4 & 2\\
6 & 2 & 2 & -6 & 1
\end{array}
\right]  _{10}$%
\end{tabular}
\
\]%
\[%
\begin{tabular}
[c]{lll}%
$\frac{1}{27}\left[
\begin{array}
[c]{rrrrr}%
20 & -12 & 10 & 9 & 2\\
-12 & 6 & 15 & 18 & \mathbf{0}\\
10 & 15 & 2 & \mathbf{0} & -20\\
9 & 18 & \mathbf{0} & \mathbf{0} & 19\\
2 & \mathbf{0} & -20 & 18 & -21
\end{array}
\right]  _{11}$ &  & $\frac{1}{21}\left[
\begin{array}
[c]{rrrrr}%
\mathbf{0} & \mathbf{0} & 4 & 5 & 20\\
\mathbf{0} & 7 & 10 & 16 & -6\\
18 & -10 & \mathbf{0} & 4 & -1\\
9 & 16 & -10 & \mathbf{0} & 2\\
6 & 6 & 15 & -12 & \mathbf{0}%
\end{array}
\right]  _{12}$%
\end{tabular}
\
\]%
\[%
\begin{tabular}
[c]{lll}%
$\frac{1}{9}\left[
\begin{array}
[c]{rrrrr}%
4 & 4 & 2 & 6 & 3\\
4 & 4 & 2 & -3 & -6\\
2 & 2 & 1 & -6 & 6\\
6 & -3 & -6 & \mathbf{0} & \mathbf{0}\\
3 & -6 & 6 & \mathbf{0} & \mathbf{0}%
\end{array}
\right]  _{13}$ &  & $\frac{1}{9}\left[
\begin{array}
[c]{rrrrr}%
8 & \mathbf{0} & -3 & 2 & 2\\
\mathbf{0} & 7 & \mathbf{0} & 4 & -4\\
-3 & \mathbf{0} & \mathbf{0} & 6 & 6\\
2 & 4 & 6 & -3 & 4\\
2 & -4 & 6 & 4 & -3
\end{array}
\right]  _{17}$%
\end{tabular}
\
\]%
\[%
\begin{tabular}
[c]{lll}%
$\frac{1}{9}\left[
\begin{array}
[c]{rrrrr}%
5 & 4 & \mathbf{0} & -6 & 2\\
4 & 3 & -4 & 6 & 2\\
\mathbf{0} & -4 & 1 & \mathbf{0} & 8\\
-6 & 6 & \mathbf{0} & \mathbf{0} & 3\\
2 & 2 & 8 & 3 & \mathbf{0}%
\end{array}
\right]  _{18}$ &  & $\frac{1}{441}\left[
\begin{array}
[c]{rrrrr}%
400 & \mathbf{0} & 100 & 105 & 116\\
\mathbf{0} & 400 & 80 & 84 & -145\\
100 & 80 & 41 & -420 & \mathbf{0}\\
105 & 84 & -420 & \mathbf{0} & \mathbf{0}\\
116 & -145 & \mathbf{0} & \mathbf{0} & -400
\end{array}
\right]  _{19}$%
\end{tabular}
\
\]%
\[%
\begin{tabular}
[c]{lll}%
$\frac{1}{6}\left[
\begin{array}
[c]{rrrrr}%
3 & 3 & 3 & -3 & \mathbf{0}\\
3 & 3 & -3 & 3 & \mathbf{0}\\
3 & -3 & 1 & 1 & 4\\
-3 & 3 & 1 & 1 & 4\\
\mathbf{0} & \mathbf{0} & 4 & 4 & -2
\end{array}
\right]  _{20}^{1}$ &  & $\frac{1}{21}\left[
\begin{array}
[c]{rrrrr}%
18 & \mathbf{0} & 6 & \mathbf{0} & -9\\
\mathbf{0} & 17 & 6 & -10 & 4\\
6 & 6 & \mathbf{0} & 15 & 12\\
\mathbf{0} & -10 & 15 & -4 & 10\\
-9 & 4 & 12 & 10 & -10
\end{array}
\right]  _{21}^{1}$%
\end{tabular}
\
\]%
\[%
\begin{tabular}
[c]{lll}%
$\frac{1}{15}\left[
\begin{array}
[c]{rrrrr}%
10 & \mathbf{0} & 3 & 4 & 10\\
\mathbf{0} & 10 & 6 & 8 & -5\\
3 & 6 & 6 & -12 & \mathbf{0}\\
4 & 8 & -12 & -1 & \mathbf{0}\\
10 & -5 & \mathbf{0} & \mathbf{0} & -10
\end{array}
\right]  _{22}$ &  & $\frac{1}{5}\left[
\begin{array}
[c]{rrrrr}%
3 & 2 & 2 & -2 & -2\\
2 & 3 & -2 & 2 & 2\\
2 & -2 & 3 & -2 & 3\\
-2 & 2 & 2 & -2 & 3\\
-2 & 2 & 2 & 3 & -2
\end{array}
\right]  _{23}$%
\end{tabular}
\
\]%
\[%
\begin{tabular}
[c]{ccc}%
$\frac{1}{11}\left[
\begin{array}
[c]{rrrrr}%
\mathbf{0} & \mathbf{0} & 2 & 6 & 9\\
1 & -2 & -6 & 8 & -4\\
2 & 7 & 6 & 4 & -4\\
4 & -8 & 6 & 1 & -2\\
10 & 2 & -3 & -2 & 2
\end{array}
\right]  _{24}$ &  & $\frac{1}{10}\left[
\begin{array}
[c]{rrrrr}%
\mathbf{0} & \mathbf{0} & \mathbf{0} & 6 & 8\\
1 & -7 & -5 & -4 & 3\\
1 & -7 & 5 & 4 & -3\\
7 & 1 & -5 & 4 & -3\\
7 & 1 & 5 & -4 & 3
\end{array}
\right]  _{25}$%
\end{tabular}
\
\]
$\ \ $%
\[%
\begin{tabular}
[c]{ccc}%
$\frac{1}{45}\left[
\begin{array}
[c]{rrrrr}%
\mathbf{0} & \mathbf{0} & \mathbf{0} & 27 & 36\\
\mathbf{0} & 42 & 6 & 12 & -9\\
5 & -16 & 12 & 32 & -24\\
20 & 2 & -39 & 8 & -6\\
40 & 1 & 18 & -8 & 6
\end{array}
\right]  _{26}$ &  & $\frac{1}{45}\left[
\begin{array}
[c]{rrrrr}%
\mathbf{0} & \mathbf{0} & \mathbf{0} & 27 & -36\\
\mathbf{0} & 35 & 20 & 16 & 12\\
15 & \mathbf{0} & -30 & 24 & 18\\
30 & -20 & 25 & 8 & 6\\
30 & 20 & -10 & -20 & -15
\end{array}
\right]  _{27}$%
\end{tabular}
\
\]
$\ \ $%
\[%
\begin{tabular}
[c]{lll}%
$\frac{1}{165}\left[
\begin{array}
[c]{rrrrr}%
\mathbf{0} & \mathbf{0} & \mathbf{0} & 132 & -99\\
\mathbf{0} & 80 & 35 & 84 & 112\\
5 & \mathbf{0} & 160 & -24 & -32\\
160 & -35 & \mathbf{0} & 12 & 16\\
-40 & -140 & 20 & 45 & 60
\end{array}
\right]  _{28}$ &  & $\frac{1}{39}\left[
\begin{array}
[c]{rrrrr}%
\mathbf{0} & \mathbf{0} & \mathbf{0} & 15 & 36\\
28 & 5 & -6 & 24 & -10\\
2 & 18 & -32 & -12 & 5\\
27 & -4 & 10 & -24 & 10\\
-2 & 34 & 19 & \mathbf{0} & \mathbf{0}%
\end{array}
\right]  _{29}$%
\end{tabular}
\
\]
$\ \ $%
\[%
\begin{tabular}
[c]{lll}%
$\frac{1}{15}\left[
\begin{array}
[c]{rrrrr}%
\mathbf{0} & \mathbf{0} & \mathbf{0} & 9 & 12\\
\mathbf{0} & 14 & 2 & 4 & -3\\
10 & -4 & 3 & 8 & -6\\
10 & 3 & 4 & -8 & 6\\
5 & 2 & -14 & \mathbf{0} & \mathbf{0}%
\end{array}
\right]  _{30}$ &  & $\frac{1}{165}\left[
\begin{array}
[c]{rrrrr}%
\mathbf{0} & \mathbf{0} & \mathbf{0} & 99 & 132\\
\mathbf{0} & 35 & 140 & 64 & -48\\
160 & \mathbf{0} & -20 & 28 & -21\\
40 & 20 & 75 & -112 & 84\\
5 & -160 & 40 & \mathbf{0} & \mathbf{0}%
\end{array}
\right]  _{31}$%
\end{tabular}
\
\]
$\ \ $%
\[%
\begin{tabular}
[c]{lll}%
$\frac{1}{15}\left[
\begin{array}
[c]{rrrrr}%
\mathbf{0} & \mathbf{0} & \mathbf{0} & 9 & 12\\
\mathbf{0} & 10 & 10 & 4 & -3\\
10 & \mathbf{0} & -5 & 8 & -6\\
10 & 5 & \mathbf{0} & -8 & 6\\
5 & -10 & 10 & \mathbf{0} & \mathbf{0}%
\end{array}
\right]  _{32}$ &  & $\frac{1}{21}\left[
\begin{array}
[c]{rrrrr}%
\mathbf{0} & \mathbf{0} & 7 & 14 & 14\\
\mathbf{0} & 7 & 0 & 14 & -14\\
-6 & 18 & 6 & -6 & 3\\
9 & 8 & -16 & 2 & 6\\
18 & 2 & 10 & -3 & -2
\end{array}
\right]  _{33}$%
\end{tabular}
\
\]
$\ \ $%
\[%
\begin{tabular}
[c]{lll}%
$\frac{1}{21}\left[
\begin{array}
[c]{rrrrr}%
\mathbf{0} & \mathbf{0} & 6 & 18 & 9\\
\mathbf{0} & 14 & \mathbf{0} & -7 & 14\\
13 & 8 & 12 & \mathbf{0} & -8\\
4 & 10 & -15 & 8 & -6\\
16 & -9 & -6 & -2 & 8
\end{array}
\right]  _{34}$ &  & $\frac{1}{21}\left[
\begin{array}
[c]{rrrrr}%
\mathbf{0} & \mathbf{0} & 6 & 9 & 18\\
\mathbf{0} & 14 & \mathbf{0} & -14 & 7\\
20 & 1 & 6 & \mathbf{0} & -2\\
4 & 10 & -15 & 10 & \mathbf{0}\\
-5 & 12 & 12 & 8 & -8
\end{array}
\right]  _{35}$%
\end{tabular}
\
\]
$\ \ $%
\[%
\begin{tabular}
[c]{lll}%
$\frac{1}{15}\left[
\begin{array}
[c]{rrrrr}%
\mathbf{0} & \mathbf{0} & \mathbf{0} & 9 & 12\\
\mathbf{0} & -12 & -9 & \mathbf{0} & \mathbf{0}\\
5 & -6 & 8 & -8 & 6\\
10 & -3 & 4 & 8 & -6\\
10 & 6 & -8 & -4 & 3
\end{array}
\right]  _{36}$ &  & $\frac{1}{13}\left[
\begin{array}
[c]{rrrrr}%
\mathbf{0} & \mathbf{0} & 4 & 3 & 12\\
1 & 10 & \mathbf{0} & -8 & 2\\
2 & 7 & 6 & 8 & -4\\
8 & 2 & -9 & 4 & 2\\
-10 & 4 & -6 & 4 & 1
\end{array}
\right]  _{37}$%
\end{tabular}
\
\]
$\ \ $%
\[%
\begin{tabular}
[c]{lll}%
$\frac{1}{15}\left[
\begin{array}
[c]{rrrrr}%
\mathbf{0} & \mathbf{0} & 5 & -10 & 10\\
3 & -12 & \mathbf{0} & 6 & 6\\
12 & 1 & 8 & \mathbf{0} & -4\\
-6 & 4 & 10 & 8 & 3\\
6 & 8 & -6 & 5 & 8
\end{array}
\right]  _{38}$ &  & $\frac{1}{33}\left[
\begin{array}
[c]{rrrrr}%
\mathbf{0} & \mathbf{0} & 11 & 22 & 22\\
8 & 15 & \mathbf{0} & -20 & 20\\
10 & 12 & 26 & \mathbf{0} & -13\\
30 & -12 & -6 & 3 & \mathbf{0}\\
5 & 24 & -16 & 14 & -6
\end{array}
\right]  _{39}$%
\end{tabular}
\
\]
$\ \ $%
\[%
\begin{tabular}
[c]{lll}%
$\frac{1}{65}\left[
\begin{array}
[c]{rrrrr}%
\mathbf{0} & \mathbf{0} & \mathbf{0} & 25 & 60\\
-7 & 60 & 24 & \mathbf{0} & \mathbf{0}\\
24 & -20 & 57 & \mathbf{0} & \mathbf{0}\\
36 & 9 & -12 & -48 & 20\\
48 & 12 & -16 & 36 & -15
\end{array}
\right]  _{40}$ &  & $\frac{1}{75}\left[
\begin{array}
[c]{rrrrr}%
\mathbf{0} & \mathbf{0} & \mathbf{0} & 45 & 60\\
\mathbf{0} & \mathbf{0} & 60 & 36 & -27\\
-10 & 70 & 15 & -16 & 12\\
50 & 25 & -30 & 32 & -24\\
55 & -10 & 30 & -32 & 24
\end{array}
\right]  _{41}$%
\end{tabular}
\
\]
$\ \ $%
\[%
\begin{tabular}
[c]{lll}%
$\frac{1}{65}\left[
\begin{array}
[c]{rrrrr}%
\mathbf{0} & \mathbf{0} & \mathbf{0} & 39 & 52\\
\mathbf{0} & \mathbf{0} & 25 & -48 & 36\\
25 & -60 & \mathbf{0} & \mathbf{0} & \mathbf{0}\\
36 & 15 & 48 & 16 & -12\\
48 & 20 & -36 & -12 & 9
\end{array}
\right]  _{42}$ &  & $\frac{1}{125}\left[
\begin{array}
[c]{rrrrr}%
\mathbf{0} & \mathbf{0} & \mathbf{0} & 75 & 100\\
\mathbf{0} & 75 & 100 & \mathbf{0} & \mathbf{0}\\
\mathbf{0} & 60 & -45 & 80 & -60\\
100 & -48 & 36 & 36 & -27\\
75 & 64 & -48 & -48 & 36
\end{array}
\right]  _{43}$%
\end{tabular}
\ .
\]
$\ $%

%TCIMACRO{\TeXButton{\endgroup}{\endgroup}}%
%BeginExpansion
\endgroup
%EndExpansion

\section{Symmetric rational orthogonal matrices\label{sec3}}

A square matrix $X$ is \emph{involutory} if $X^{2}=I$. It is well known that a
matrix $X\in U(n)$ is hermitian if and only if it is involutory. In
particular, a matrix $X\in O(n)$ is symmetric if and only if it is involutory.

One can easily formulate the symmetric analogues of Conjecture \ref{con1} and
Conjecture \ref{con2}. For the sake of simplicity we shall formulate just the
combined conjecture.

\begin{conjecture}
\label{con3}For any hermitian $X\in U(n)$ there exists a symmetric $Z\in
O_{n}(\mathbb{Q})$ such that $\underline{X}=\underline{Z}$ and $\mathrm{Tr}%
\left(  X\right)  =\mathrm{Tr}\left(  Z\right)  $.
\end{conjecture}

Let $X=X^{\dagger}\in U(n)$. Then $X^{2}=I_{n}$ and so the eigenvalues of $X$
belong to $\left\{  \pm1\right\}  $. Consequently, Tr$\left(  X\right)  $ is
an integer congruent to $n\operatorname{mod}2$. Since $\underline
{-X}=\underline{X}$ and Tr$\left(  -X\right)  =$ $-$Tr$\left(  X\right)  $, in
proving this conjecture we may assume that Tr$\left(  X\right)  \geq0$.
Clearly, we can also assume that the zero-pattern $\underline{X}$ is
indecomposable (we also know that it is necessarily SQ). There are further
restrictions on possible values of the trace.

\begin{proposition}
\label{n-2}There is no indecomposable hermitian matrix $X\in U(n)$, $n\geq2$,
with $X_{1,2}=0$ and $\mathrm{Tr}\left(  X\right)  =n-2$.
\end{proposition}

\begin{proof}
Suppose that such a matrix, $X$, exists. As $X^{2}=I$ and $X\neq\pm I$, the
eigenvalues of $X$ are $+1$ and $-1$ and the two eigenspaces of $X$ are
orthogonal to each other. By indecomposability we have $X_{2,2}\neq1$. Let
$\{e_{1},\ldots,e_{n}\}$ be the standard basis of $\mathbb{C}^{n}$. Since
$X_{1,2}=0$, the vector $Xe_{2}$ is orthogonal to $e_{1}$ and also $Xe_{2}\neq
e_{2}$. Thus the vector $v=Xe_{2}-e_{2}$ is nonzero and $v\perp e_{1}$. As
$Xv=-v$ and the $-1$-eigenspace of $X$ is 1-dimensional, we conclude that the
subspace $v^{\perp}$ is the $+1$-eigenspace of $X$. Hence $Xe_{1}=e_{1}$,
\emph{i.e.}, $X_{1,1}=1$. This contradicts the indecomposability of $X$.
\end{proof}

\bigskip

The objective of this section is to provide a support for the above conjecture
by constructing examples of symmetric rational orthogonal matrices with
specified indecomposable zero-pattern and specified trace. We shall consider
zero-patterns of size $n\leq5$.

Two symmetric $(0,1)$-matrices $X$ and $Y$ are said to be \emph{congruent }if
there is a permutation matrix $P$ such that $PXP^{T}=Y$. In graph-theoretical
terms, the permutation matrix $P$ represents an isomorphism between the
undirected graph with adjacency matrix $X$ and the undirected graph with
adjacency matrix $Y$.

Let $X=[X_{i,j}]\in U(n)$ be a hermitian matrix. We say that $X$ is in
\emph{quasi-normal form} if Tr$(X)\geq0$ and $X_{1,1}\geq X_{2,2}\geq
\cdots\geq X_{n,n}$. In our list a matrix $X$ will be written in the form%
\[
X=\frac{1}{d}\left[
\begin{array}
[c]{ccc}%
\ast & \cdots & \ast\\
\vdots & \ddots & \vdots\\
\ast & \cdots & \ast
\end{array}
\right]  _{k,l}^{t},
\]
where $k$ and $l$ are simply numerical labels and $t$ is the trace of the
matrix. The index $k$ corresponds to the one used for the matrices in Section
\ref{sec2}. The index $l$ specifies the congruence class of symmetric
zero-patterns within the $k$-th equivalence class.

Our list is not complete. We are in fact unable to construct symmetric
rational orthogonal matrices with specified trace for exactly two among all
zero-patterns. For these matrices, we give examples of matrices \emph{as}
\emph{close as possible} to symmetric rational one in Section \ref{sec5}. All
denominators in the list are minimal except for a few cases, when $n=5$.
Exceptions are the cases $(k,l)=\left(  2,2\right)  ,\left(  3,2\right)
,\left(  6,1\right)  ,(11,2),(22,2)$.

\subsection{$n=2$%
\[%
\protect\begin{tabular}
[c]{l}%
$\frac{1}{5}\left[
\protect\begin{array}
[c]{rr}%
3 & 4\\
4 & -3
\protect\end{array}
\right]  _{1}^{0}$%
\protect\end{tabular}
.
\]
}

\subsection{$n=3$}%

%TCIMACRO{\TeXButton{\begingroup}{\begingroup}}%
%BeginExpansion
\begingroup
%EndExpansion%
%TCIMACRO{\TeXButton{\scalefont{x}}{\scalefont{1}}}%
%BeginExpansion
\scalefont{1}%
%EndExpansion%
\[%
\begin{tabular}
[c]{lll}%
$\frac{1}{25}\left[
\begin{array}
[c]{rrr}%
16 & 12 & 15\\
12 & 9 & -20\\
15 & -20 & \mathbf{0}%
\end{array}
\right]  _{1}^{1}$ &  & $\frac{1}{3}\left[
\begin{array}
[c]{rrr}%
2 & -1 & 2\\
-1 & 2 & 2\\
2 & 2 & -1
\end{array}
\right]  _{2}^{1}$%
\end{tabular}
.
\]
%

%TCIMACRO{\TeXButton{\endgroup}{\endgroup}}%
%BeginExpansion
\endgroup
%EndExpansion

\subsection{$n=4$}%

%TCIMACRO{\TeXButton{\begingroup}{\begingroup}}%
%BeginExpansion
\begingroup
%EndExpansion%
%TCIMACRO{\TeXButton{\scalefont{x}}{\scalefont{1}}}%
%BeginExpansion
\scalefont{1}%
%EndExpansion

\begin{description}
\item
\[%
\begin{tabular}
[c]{ccc}%
$\frac{1}{9}\left[
\begin{array}
[c]{rrrr}%
8 & -3 & 2 & 2\\
-3 & \mathbf{0} & 6 & 6\\
2 & 6 & -4 & 5\\
2 & 6 & 5 & -4
\end{array}
\right]  _{1}^{0}$ &  & $\frac{1}{9}\left[
\begin{array}
[c]{rrrr}%
8 & 2 & -2 & 3\\
2 & 5 & 4 & -6\\
-2 & 4 & 5 & 6\\
3 & -6 & 6 & \mathbf{0}%
\end{array}
\right]  _{1}^{2}$%
\end{tabular}
\
\]%
\[%
\begin{tabular}
[c]{ccc}%
$\frac{1}{9}\left[
\begin{array}
[c]{rrrr}%
6 & \mathbf{0} & 3 & -6\\
\mathbf{0} & 1 & 8 & 4\\
3 & 8 & -2 & 2\\
-6 & 4 & 2 & -5
\end{array}
\right]  _{2,1}^{0}$ &  & $\frac{1}{9}\left[
\begin{array}
[c]{rrrr}%
8 & 5 & -10 & 6\\
5 & \mathbf{0} & 10 & 10\\
-10 & 10 & \mathbf{0} & 5\\
6 & 10 & 5 & -8
\end{array}
\right]  _{2,2}^{0}$%
\end{tabular}
\
\]%
\[%
\begin{tabular}
[c]{lll}%
$\frac{1}{33}\left[
\begin{array}
[c]{rrrr}%
16 & 7 & \mathbf{0} & -28\\
7 & \mathbf{0} & 32 & 4\\
\mathbf{0} & 32 & -1 & 8\\
-28 & 4 & 8 & -15
\end{array}
\right]  _{3}^{0}$ &  & $\frac{1}{3}\left[
\begin{array}
[c]{rrrr}%
2 & \mathbf{0} & 2 & 1\\
\mathbf{0} & 2 & 1 & -2\\
2 & 1 & -2 & \mathbf{0}\\
1 & -2 & \mathbf{0} & -2
\end{array}
\right]  _{4,1}^{0}$%
\end{tabular}
\]%
\[%
\begin{tabular}
[c]{ccc}%
$\frac{1}{3}\left[
\begin{array}
[c]{rrrr}%
1 & 2 & -2 & \mathbf{0}\\
2 & \mathbf{0} & 1 & 2\\
-2 & 1 & \mathbf{0} & 2\\
\mathbf{0} & 2 & 2 & -1
\end{array}
\right]  _{4,2}^{0}$ &  & $\frac{1}{65}\left[
\begin{array}
[c]{rrrr}%
25 & \mathbf{0} & -36 & 48\\
\mathbf{0} & \mathbf{0} & 52 & 39\\
-36 & 52 & -9 & 12\\
48 & 39 & 12 & -16
\end{array}
\right]  _{5}^{0}$%
\end{tabular}
\]%
\[%
\begin{tabular}
[c]{lll}%
$\frac{1}{2}\left[
\begin{array}
[c]{rrrr}%
1 & -1 & 1 & 1\\
-1 & 1 & 1 & 1\\
1 & 1 & -1 & 1\\
1 & 1 & 1 & -1
\end{array}
\right]  _{8}^{0}$ &  & $\frac{1}{2}\left[
\begin{array}
[c]{rrrr}%
1 & 1 & 1 & -1\\
1 & 1 & -1 & 1\\
1 & -1 & 1 & 1\\
-1 & 1 & 1 & 1
\end{array}
\right]  _{8}^{2}$%
\end{tabular}
.
\]

\end{description}%

%TCIMACRO{\TeXButton{\endgroup}{\endgroup}}%
%BeginExpansion
\endgroup
%EndExpansion

\subsection{$n=5$}%

%TCIMACRO{\TeXButton{\begingroup}{\begingroup}}%
%BeginExpansion
\begingroup
%EndExpansion%
%TCIMACRO{\TeXButton{\scalefont{x}}{\scalefont{1}}}%
%BeginExpansion
\scalefont{1}%
%EndExpansion

\begin{description}
\item
\[%
\begin{tabular}
[c]{ccc}%
$\frac{1}{4}\left[
\begin{array}
[c]{rrrrr}%
3 & 1 & 1 & 1 & -2\\
1 & 3 & -1 & -1 & 2\\
1 & -1 & 3 & -1 & 2\\
1 & -1 & -1 & 3 & 2\\
-2 & 2 & 2 & 2 & \mathbf{0}%
\end{array}
\right]  _{1}^{3}$ &  & $\frac{1}{27}\left[
\begin{array}
[c]{rrrrr}%
19 & 4 & -12 & 12 & 8\\
4 & 16 & 6 & -15 & 14\\
-12 & 6 & 9 & 18 & 12\\
12 & -15 & 18 & \mathbf{0} & 6\\
8 & 14 & 12 & 6 & -17
\end{array}
\right]  _{1}^{1}$%
\end{tabular}
\ \ \ \
\]%
\[%
\begin{tabular}
[c]{ccc}%
$\frac{1}{7}\left[
\begin{array}
[c]{rrrrr}%
4 & -3 & 2 & 4 & 2\\
-3 & 4 & 2 & 4 & 2\\
2 & 2 & 3 & -4 & 4\\
4 & 4 & -4 & 1 & \mathbf{0}\\
2 & 2 & 4 & \mathbf{0} & -5
\end{array}
\right]  _{2,1}^{1}$ &  & $\frac{1}{375}\left[
\begin{array}
[c]{rrrrr}%
200 & 90 & 120 & 125 & -250\\
90 & 168 & -276 & 150 & 75\\
120 & -276 & 7 & 200 & 100\\
125 & 150 & 200 & \mathbf{0} & 250\\
-250 & 75 & 100 & 250 & \mathbf{0}%
\end{array}
\right]  _{2,2}^{1}$%
\end{tabular}
\ \ \
\]%
\[%
\begin{tabular}
[c]{ccc}%
$\frac{1}{11}\left[
\begin{array}
[c]{rrrrr}%
8 & 4 & 1 & 2 & -6\\
4 & 5 & -4 & \mathbf{0} & 8\\
1 & -4 & \mathbf{0} & 10 & 2\\
2 & \mathbf{0} & 10 & -1 & 4\\
-6 & 8 & 2 & 4 & -1
\end{array}
\right]  _{3,1}^{1}$ &  & $\frac{1}{325}\left[
\begin{array}
[c]{rrrrr}%
245 & 84 & 80 & -140 & 112\\
84 & 80 & 112 & 35 & -280\\
80 & 112 & \mathbf{0} & 280 & 91\\
-140 & 35 & 280 & \mathbf{0} & 80\\
112 & -280 & 91 & 80 & \mathbf{0}%
\end{array}
\right]  _{3,2}^{1}$%
\end{tabular}
\ \ \
\]%
\[%
\begin{tabular}
[c]{ccc}%
$\frac{1}{5}\left[
\begin{array}
[c]{rrrrr}%
3 & 2 & -2 & 2 & 2\\
2 & 2 & 1 & -4 & \mathbf{0}\\
-2 & 1 & 2 & \mathbf{0} & 4\\
2 & -4 & \mathbf{0} & -1 & 2\\
2 & \mathbf{0} & 4 & 2 & -1
\end{array}
\right]  _{4,1}^{1}$ &  & $\frac{1}{5}\left[
\begin{array}
[c]{rrrrr}%
4 & \mathbf{0} & 1 & -2 & 2\\
\mathbf{0} & 4 & -2 & 1 & 2\\
1 & -2 & \mathbf{0} & 4 & 2\\
-2 & 1 & 4 & \mathbf{0} & 2\\
2 & 2 & 2 & 2 & -3
\end{array}
\right]  _{4,2}^{1}$%
\end{tabular}
\ \ \ \
\]%
\[%
\begin{tabular}
[c]{lll}%
$\frac{1}{147}\left[
\begin{array}
[c]{rrrrr}%
145 & 8 & \mathbf{0} & 14 & -18\\
8 & 51 & 80 & -112 & \mathbf{0}\\
\mathbf{0} & 80 & 47 & 70 & 90\\
14 & -112 & 70 & \mathbf{0} & 63\\
-18 & \mathbf{0} & 90 & 63 & -96
\end{array}
\right]  _{5,1}^{1}$ &  & $\frac{1}{625}\left[
\begin{array}
[c]{rrrrr}%
256 & 240 & 192 & 375 & 300\\
240 & 225 & 180 & \mathbf{0} & -500\\
192 & 180 & 144 & -500 & 225\\
375 & \mathbf{0} & -500 & \mathbf{0} & \mathbf{0}\\
300 & -500 & 225 & \mathbf{0} & \mathbf{0}%
\end{array}
\right]  _{6}^{1}$%
\end{tabular}
\ \ \
\]%
\[%
\begin{tabular}
[c]{lll}%
$\frac{1}{75}\left[
\begin{array}
[c]{rrrrr}%
50 & -25 & \mathbf{0} & 30 & 40\\
-25 & 50 & \mathbf{0} & 30 & 40\\
\mathbf{0} & \mathbf{0} & \mathbf{0} & 60 & -45\\
30 & 30 & 60 & -9 & -12\\
40 & 40 & -45 & -12 & -16
\end{array}
\right]  _{7}^{1}$ &  & $\frac{1}{25}\left[
\begin{array}
[c]{ccccc}%
16 & 12 & \mathbf{0} & 12 & -9\\
12 & 9 & \mathbf{0} & -16 & 12\\
\mathbf{0} & \mathbf{0} & \mathbf{0} & 15 & 20\\
12 & -16 & 15 & \mathbf{0} & \mathbf{0}\\
-9 & 12 & 20 & \mathbf{0} & \mathbf{0}%
\end{array}
\right]  _{8}^{1}$%
\end{tabular}
\ \ \
\]%
\[%
\begin{tabular}
[c]{lll}%
$\frac{1}{9}\left[
\begin{array}
[c]{rrrrr}%
8 & 2 & 2 & \mathbf{0} & -3\\
2 & 4 & -6 & 3 & 4\\
2 & -6 & 1 & 6 & 2\\
\mathbf{0} & 3 & 6 & \mathbf{0} & 6\\
-3 & 4 & 2 & 6 & -4
\end{array}
\right]  _{9}^{1}$ &  & $\frac{1}{27}\left[
\begin{array}
[c]{rrrrr}%
19 & 12 & 8 & -12 & 4\\
12 & 5 & -20 & 12 & 4\\
8 & -20 & 3 & \mathbf{0} & 16\\
-12 & 12 & \mathbf{0} & \mathbf{0} & 21\\
4 & 4 & 16 & 21 & \mathbf{0}%
\end{array}
\right]  _{10}^{1}$%
\end{tabular}
\ \ \
\]%
\[
\frac{1}{27}\left[
\begin{array}
[c]{rrrrr}%
20 & -12 & 10 & 9 & 2\\
-12 & 6 & 15 & 18 & \mathbf{0}\\
10 & 15 & 2 & \mathbf{0} & -20\\
9 & 18 & \mathbf{0} & \mathbf{0} & 18\\
2 & \mathbf{0} & -20 & 18 & -1
\end{array}
\right]  _{11,1}^{1}%
\]%
\[
\frac{1}{78625}\left[
\begin{array}
[c]{rrrrr}%
50320 & 27156 & \mathbf{0} & -46620 & 27183\\
27156 & 28305 & -43680 & 51408 & 9620\\
\mathbf{0} & -43680 & \mathbf{0} & 12025 & 64260\\
-46620 & 51408 & 12025 & \mathbf{0} & 34944\\
27183 & 9620 & 64260 & 34944 & \mathbf{0}%
\end{array}
\right]  _{11,2}^{1}%
\]%
\[%
\begin{tabular}
[c]{lll}%
$\frac{1}{9}\left[
\begin{array}
[c]{rrrrr}%
4 & 4 & 2 & 6 & 3\\
4 & 4 & 2 & -3 & -6\\
2 & 2 & 1 & -6 & 6\\
6 & -3 & -6 & \mathbf{0} & \mathbf{0}\\
3 & -6 & 6 & \mathbf{0} & \mathbf{0}%
\end{array}
\right]  _{13}^{1}$ &  & $\frac{1}{9}\left[
\begin{array}
[c]{rrrrr}%
8 & \mathbf{0} & -3 & 2 & 2\\
\mathbf{0} & 7 & \mathbf{0} & 4 & -4\\
-3 & \mathbf{0} & \mathbf{0} & 6 & 6\\
2 & 4 & 6 & -3 & 4\\
2 & -4 & 6 & 4 & -3
\end{array}
\right]  _{17}^{1}$%
\end{tabular}
\ \ \
\]
$\ \ $%
\[%
\begin{tabular}
[c]{lll}%
$\frac{1}{9}\left[
\begin{array}
[c]{rrrrr}%
5 & 4 & \mathbf{0} & -6 & 2\\
4 & 3 & -4 & 6 & 2\\
\mathbf{0} & -4 & 1 & \mathbf{0} & 8\\
-6 & 6 & \mathbf{0} & \mathbf{0} & 3\\
2 & 2 & 8 & 3 & \mathbf{0}%
\end{array}
\right]  _{18}^{1}$ &  & $\frac{1}{441}\left[
\begin{array}
[c]{rrrrr}%
400 & \mathbf{0} & 100 & 105 & 116\\
\mathbf{0} & 400 & 80 & 84 & -145\\
100 & 80 & 41 & -420 & \mathbf{0}\\
105 & 84 & -420 & \mathbf{0} & \mathbf{0}\\
116 & -145 & \mathbf{0} & \mathbf{0} & -400
\end{array}
\right]  _{19}^{1}$%
\end{tabular}
\ \ \
\]%
\[%
\begin{tabular}
[c]{lll}%
$\frac{1}{6}\left[
\begin{array}
[c]{rrrrr}%
3 & 3 & 3 & -3 & \mathbf{0}\\
3 & 3 & -3 & 3 & \mathbf{0}\\
3 & -3 & 1 & 1 & 4\\
-3 & 3 & 1 & 1 & 4\\
\mathbf{0} & \mathbf{0} & 4 & 4 & -2
\end{array}
\right]  _{20}^{1}$ &  & $\frac{1}{21}\left[
\begin{array}
[c]{rrrrr}%
18 & \mathbf{0} & 6 & \mathbf{0} & -9\\
\mathbf{0} & 17 & 6 & -10 & 4\\
6 & 6 & \mathbf{0} & 15 & 12\\
\mathbf{0} & -10 & 15 & -4 & 10\\
-9 & 4 & 12 & 10 & -10
\end{array}
\right]  _{21}^{1}$%
\end{tabular}
\ \ \
\]%
\[%
\begin{tabular}
[c]{ccc}%
$\frac{1}{15}\left[
\begin{array}
[c]{rrrrr}%
10 & \mathbf{0} & 3 & 4 & 10\\
\mathbf{0} & 10 & 6 & 8 & -5\\
3 & 6 & 6 & -12 & \mathbf{0}\\
4 & 8 & -12 & -1 & \mathbf{0}\\
10 & -5 & \mathbf{0} & \mathbf{0} & -10
\end{array}
\right]  _{22,1}^{1}$ &  & $\frac{1}{75}\left[
\begin{array}
[c]{rrrrr}%
39 & \mathbf{0} & -48 & 30 & 30\\
\mathbf{0} & 25 & \mathbf{0} & -50 & 50\\
-48 & \mathbf{0} & 11 & 40 & 40\\
30 & -50 & 40 & \mathbf{0} & 25\\
30 & 50 & 40 & 25 & \mathbf{0}%
\end{array}
\right]  _{22,2}^{1}$%
\end{tabular}
\ \ \
\]
$\ \ $%
\[%
\begin{tabular}
[c]{lll}%
$\frac{1}{5}\left[
\begin{array}
[c]{rrrrr}%
3 & 2 & 2 & -2 & -2\\
2 & 3 & -2 & 2 & 2\\
2 & -2 & 3 & -2 & 2\\
-2 & 2 & 2 & -2 & 3\\
-2 & 2 & 2 & 3 & -2
\end{array}
\right]  _{23}^{1}$ &  & $\frac{1}{5}\left[
\begin{array}
[c]{rrrrr}%
3 & 2 & 2 & 2 & -2\\
2 & 3 & -2 & -2 & 2\\
2 & -2 & 3 & -2 & 2\\
2 & -2 & -2 & 3 & 2\\
-2 & 2 & 2 & 2 & 3
\end{array}
\right]  _{23}^{3}$%
\end{tabular}
\ \ \ .
\]
$\ $
\end{description}

%

%TCIMACRO{\TeXButton{\endgroup}{\endgroup}}%
%BeginExpansion
\endgroup
%EndExpansion

\section{Infinite families of rational orthogonal matrices\label{sec4}}

In this section we employ different techniques to construct infinite families
of symmetric rational orthogonal matrices with specified zero-pattern and
trace. The following well-known fact will be useful. We include a proof for
the sake of completeness.

\begin{proposition}
\label{dense}The set $SO_{n}(\mathbb{Q})$ is dense in $SO(n)$ (in Euclidean topology).
\end{proposition}

\begin{proof}
It is sufficient to observe that the Cayley transformation
\[
X\longmapsto Y=\frac{I+X}{I-X},
\]
from the space of $n\times n$ real skew-symmetric matrices to $SO(n)$ has
dense image, and if $X$ is a rational matrix so is $Y$.
\end{proof}

\bigskip

Let $\Delta_{n,k}$ be the $n\times n$ zero-pattern all of whose entries are
$1$ except for the first $k$ diagonal entries which are $0$.

\begin{corollary}
Let $0\leq k<l<n$. If there exists $X=X^{T}\in O_{n}(\mathbb{Q})$ with
$\underline{X}=\Delta_{n,l}$, then there exists $Y=Y^{T}\in O_{n}(\mathbb{Q})$
with $\underline{X}=\Delta_{n,k}$ and Tr$\left(  X\right)  =$ Tr$\left(
Y\right)  $.
\end{corollary}

\begin{proof}
Without any loss of generality we may assume that $l=k+1$. Let $X=X^{T}\in
O_{n}(\mathbb{Q})$ be such that $\underline{X}=\Delta_{n,l}$. Set $Y=P\times
P^{T}$, where $P=I_{k}\oplus R\oplus I_{n-l-1}$ and $R$ is the rotation matrix%
\[
\left[
\begin{array}
[c]{rr}%
\cos\theta & -\sin\theta\\
\sin\theta & \cos\theta
\end{array}
\right]  .
\]
Clearly we can choose $\theta\in\mathbb{R}$ such that $\underline{Y}%
=\Delta_{n,k}$. Since the rational points on the unit circle are dense (see
Proposition \ref{dense}), we can replace $R$ with $R_{1}\in SO_{2}%
(\mathbb{Q})$ without affecting the zero-pattern of $Y$.
\end{proof}

\subsection{Symmetric rational orthogonal matrices with few zero entries}

Observe that the matrix $X_{n}=I_{n}-\frac{2}{n}J_{n}$ is rational orthogonal
and involutory, where $J_{n}$ denotes the all-ones matrix. Moreover,
Tr$\left(  X_{n}\right)  =n-2$ and, if $n>2$, $X_{n}$ has no zero entries,
\emph{i.e.}, $\underline{X}=J_{n}$. This $X_{n}$ is often called \emph{Grover
matrix} in the literature of quantum computation (see, \emph{e.g.},
\cite{nc00}).

\begin{proposition}
Let $t=n-2k$ where $k\in\left\{  1,2,...,n-1\right\}  $. Then there exists a
symmetric matrix $X\in O_{n}(\mathbb{Q})$ such that Tr$\left(  X\right)  =t$
and $\underline{X}=J_{n}$.
\end{proposition}

\begin{proof}
Noe that the assertion is vacuous for $n=1$ and trivial for $n=2$. We proceed
by induction on $n\geq3$. We may assume that $t\geq0$. If $k=1$ the above
observation shows that the assertion is true. Let $k>1$. Then $t=n-2k\leq n-4$
implies that $n\geq4$. By induction hypothesis there exists a symmetric matrix
$Y\in O_{n-2}(\mathbb{Q})$ such that Tr$\left(  Y\right)  =t$ and
$\underline{Y}=J_{n-2}$. The matrix
\[
Z=Y\oplus\frac{1}{5}\left[
\begin{array}
[c]{rr}%
3 & 4\\
4 & -3
\end{array}
\right]  \in O_{n}(\mathbb{Q})
\]
is symmetric with Tr$\left(  Z\right)  =t$. By using Proposition \ref{dense},
we can choose $P\in O_{n}(\mathbb{Q})$ such that $X=PZP^{T}$ has no zero
entries, \emph{i.e.}, $\underline{X}=J_{n}$.
\end{proof}

\begin{proposition}
Let $X=X^{T}\in O_{n}(\mathbb{Q})$ be such that $X_{i,n}\neq0$ for $1\leq
i\leq n$. Then, $m>n>1$, there exists $Y=Y^{T}\in O_{m}(\mathbb{Q})$ such that
$X_{i,j}=0$ if and only if $Y_{i,j}=0$, for $1\leq i,j\leq n$, and
$Y_{i,j}\neq0$, for $i>n$. Moreover, $Y$ can be chosen so that Tr$\left(
Y\right)  =m-n+$Tr$\left(  X\right)  $.
\end{proposition}

\begin{proof}
Without any loss of generality we may assume that $m=n+1$. Then we can take
$Y=P\left(  X\oplus\left[  1\right]  \right)  P^{T}$, where%
\[
P=I_{n-1}\oplus\left[
\begin{array}
[c]{rr}%
a & b\\
b & -a
\end{array}
\right]
\]
and $a,b\in\mathbb{Q}^{\ast}$ are chosen such that $a^{2}+b^{2}=1$ and
$a^{2}/b^{2}\neq-X_{n,n}^{\pm1}$.
\end{proof}

\bigskip

By using the fact that $\Delta_{4,2}$ supports a matrix $X=X^{T}\in
O_{4}(\mathbb{Q})$ such that Tr$\left(  X\right)  =0$, if follows from the
above proposition that $\Delta_{m,2}$, $m\geq4$, supports a matrix $Y=Y^{T}\in
O_{m}(\mathbb{Q})$ with Tr$\left(  Y\right)  =m-4$.

\subsection{Symmetric rational orthogonal matrices with zero-pattern
$J_{n}-I_{n}$}

If there exists a symmetric matrix $X\in O(n)$ with zero-pattern $J_{n}-I_{n}%
$, then $n$ must be even. Indeed since such $X$ is involutory, its trace is an
integer of the same parity as $n$.

A \emph{conference matrix} of order $n$ is an $n\times n$ matrix $C$ with zero
diagonal and all other entries in $\{\pm1\}$ and such that $CC^{T}=(n-1)I_{n}%
$. If a conference matrix of order $n>1$ exists, then $n$ must be even. It is
known that they exist for all even orders $n=2m\leq64$ except for $m=11,17$
and $29$ (when they do not exist). A conference matrix is \emph{normalized} if
all entries in the first row and column are equal to 1, except the $(1,1)$
entry which is $0$.

Let $C$ be a normalized conference matrix of order $n$. If $n\equiv2\pmod{4}$,
then $C$ is necessarily symmetric. On the other hand, if $n\equiv0\pmod{4}$,
then the submatrix of $C$ obtained by deleting the first row and column is
necessarily skew-symmetric. By a well known construction of Paley (see,
\emph{e.g.} \cite{gs}), we know that there exist conference matrices of order
$n=1+p^{k}$ for any odd prime $p$ and any positive integer $k$. From these
facts we deduce the following result.

\begin{proposition}
\label{conf}Let $C$ be a normalized conference matrix of order $n=1+m^{2}$,
where $m$ is an odd positive integer. Then $\frac{1}{m}C$ is a symmetric
rational orthogonal matrix with zero-pattern $J_{n}-I_{n}$. Such $C$ exists if
$m$ is an odd prime power.
\end{proposition}

\subsection{Symmetric rational orthogonal matrices from orthogonal designs}

An \emph{orthogonal design }(see, \emph{e.g.}, \cite{gse}) of order $n$ and
type $\left(  s_{1},s_{2},...,s_{u}\right)  $ for $s_{i}>0$, on the commuting
variables $x_{1},x_{2},...,x_{u}$, is an $n\times n$ matrix $M$ with entries
from $\{0,\pm x_{i}:i=1,2,...,u\}$ such that%
\[
MM^{T}=\left(
%TCIMACRO{\dsum \limits_{i=1}^{u}}%
%BeginExpansion
{\displaystyle\sum\limits_{i=1}^{u}}
%EndExpansion
s_{i}x_{i}^{2}\right)  I_{n}.
\]
Such design can be used to construct infinitely many rational orthogonal
matrices with the same zero-pattern. As an example, consider the following
orthogonal design:%

\[
X=\left[
\begin{array}
[c]{rrrrrrrr}%
x & y & z & 0 & a & 0 & 0 & -b\\
y & -x & 0 & -z & 0 & -a & b & 0\\
z & 0 & -x & y & 0 & -b & -a & 0\\
0 & -z & y & x & b & 0 & 0 & a\\
a & 0 & 0 & b & -x & y & z & 0\\
0 & -a & -b & 0 & y & x & 0 & -z\\
0 & b & -a & 0 & z & 0 & x & y\\
-b & 0 & 0 & a & 0 & -z & y & -x
\end{array}
\right]  ,
\]%
\[
XX^{T}=\left(  x^{2}+y^{2}+z^{2}+a^{2}+b^{2}\right)  I.
\]
If we set $x=y=z=1/4$, $a=1/2$, and $b=3/4$, we then obtain a symmetric matrix
in $O_{8}(\mathbb{Q})$, with the same zero-pattern as $X$.

\subsection{Indecomposable rational orthogonal matrices with maximal number of
zero entries}

We recall from \cite{bb} that the maximum number of the zero entries in an
indecomposable $n\times n$ unitary matrix, $n\geq2$, is $\left(  n-2\right)
^{2}$. Let us say that an indecomposable $n\times n$ zero-pattern is
\emph{maximal} if it has exactly $\left(  n-2\right)  ^{2}$ zero entries. In
the same paper it is shown that, for $n\geq5$, the maximal zero-patterns form
either a single equivalence class or two equivalence classes which are
transposes of each other. We shall see below that both possibilities occur. It
is also known (see \cite{ch}) that the number of zero entries in
indecomposable $n\times n$ unitary matrices can take any of the values
$0,1,2,...,\left(  n-2\right)  ^{2}$.

We shall use the \emph{special zigzag} matrices introduced in \cite{DZ}. These
are the matrices $X$ defined by means of two sequences $x_{0},x_{1}%
,x_{2},\ldots$ and $y_{1},y_{2},\ldots$ as follows:
\begin{equation}
X=\left[
\begin{tabular}
[c]{rrrrrrrr}%
$x_{0}x_{1}$ & $x_{0}y_{1}$ & \multicolumn{1}{|r}{$\mathbf{0}$} &
\multicolumn{1}{|r}{$\mathbf{0}$} & \multicolumn{1}{|r}{$\mathbf{0}$} &
\multicolumn{1}{|r}{$\mathbf{0}$} & \multicolumn{1}{|r}{$\mathbf{0}$} &
\multicolumn{1}{|r}{$\cdots$}\\
$-y_{1}x_{2}$ & $x_{1}x_{2}$ & \multicolumn{1}{|r}{$y_{2}x_{3}$} &
\multicolumn{1}{|r}{$y_{2}y_{3}$} & \multicolumn{1}{|r}{$\mathbf{0}$} &
\multicolumn{1}{|r}{$\mathbf{0}$} & \multicolumn{1}{|r}{$\mathbf{0}$} &
\multicolumn{1}{|r}{}\\\cline{1-2}%
$y_{1}y_{2}$ & $-x_{1}y_{2}$ & $x_{2}x_{3}$ & \multicolumn{1}{|r}{$x_{2}y_{3}%
$} & \multicolumn{1}{|r}{$\mathbf{0}$} & \multicolumn{1}{|r}{$\mathbf{0}$} &
\multicolumn{1}{|r}{$\mathbf{0}$} & \multicolumn{1}{|r}{}\\\cline{1-3}%
$\mathbf{0}$ & $\mathbf{0}$ & $-y_{3}x_{4}$ & $x_{3}x_{4}$ &
\multicolumn{1}{|r}{$y_{4}x_{5}$} & \multicolumn{1}{|r}{$y_{4}y_{5}$} &
\multicolumn{1}{|r}{$\mathbf{0}$} & \multicolumn{1}{|r}{}\\\cline{1-4}%
$\mathbf{0}$ & $\mathbf{0}$ & $y_{3}y_{4}$ & $-x_{3}y_{4}$ & $x_{4}x_{5}$ &
\multicolumn{1}{|r}{$x_{4}y_{5}$} & \multicolumn{1}{|r}{$\mathbf{0}$} &
\multicolumn{1}{|r}{}\\\cline{1-5}%
$\mathbf{0}$ & $\mathbf{0}$ & $\mathbf{0}$ & $\mathbf{0}$ & $-y_{5}x_{6}$ &
$x_{5}x_{6}$ & \multicolumn{1}{|r}{$y_{6}x_{7}$} & \multicolumn{1}{|r}{}%
\\\cline{1-6}%
$\mathbf{0}$ & $\mathbf{0}$ & $\mathbf{0}$ & $\mathbf{0}$ & $y_{5}y_{6}$ &
$-x_{5}y_{6}$ & $x_{6}x_{7}$ & \multicolumn{1}{|r}{}\\\cline{1-7}%
$\vdots$ &  &  &  &  &  &  & $\ddots$%
\end{tabular}
\right]  . \label{szig}%
\end{equation}

If the above sequences are infinite, $X$ will be an infinite matrix and we
shall denote it by $X_{\infty}$. If $X$ is of size $n$ then we shall denote it
by $X_{n}$. Thus $X_{n}$ is defined by two finite sequences: $x_{0}%
,x_{1},\ldots,x_{n}$ and $y_{1},y_{2},\ldots,y_{n-1}$. Note that $X_{n}$ is
just the $n\times n$ submatrix lying in the left upper corner of $X_{\infty}$.
If $x_{k}^{2}+y_{k}^{2}=1$ for $1\leq k\leq n-1$ and $x_{0},x_{n}\in\{\pm1\}$,
then $X_{n}\in O(n)$.

\begin{proposition}
If $\underline{M}$ is a maximal $n\times n$ zero-pattern then there exists
$X\in O_{n}(\mathbb{Q})$ such that $\underline{X}=\underline{M}$.
\end{proposition}

\begin{proof}
In the above matrix $X_{n}$, we can chose the rational values for $x_{k}$ and
$y_{k}$ such that $x_{k}^{2}+y_{k}^{2}=1$ and $x_{k}y_{k}\neq0$, for $1\leq
k\leq n-1$, and set $x_{0}=x_{n}=1$. Then $X_{n}\in O_{n}(\mathbb{Q})$ and it
has the desired zero-pattern. It remains to observe that $\underline{X_{n}}$
must be equivalent to $\underline{M}$ or $\underline{M}^{T}$ by a result of
\cite{bb}.
\end{proof}

\bigskip

Let $Y_{n}$ be the matrix obtained from $X_{n}$ by reversing the order of its
rows. Let us denote its zero-pattern by $\Lambda_{n}$. This is an example of a
maximal zero-pattern (see \cite{bb}).

We set $x_{0}=x_{n}=1$ and impose the conditions $x_{k}=x_{n-k}$,
$y_{k}=y_{n-k}$ and $x_{k}^{2}+y_{k}^{2}=1$ for $1\leq k<n$. We can choose
such $x_{k}$ and $y_{k}$ to be rational and nonzero. Hence, in that case
maximal zero-patterns form just one equivalence class. If $n$ is odd, then
$\Lambda_{n}$ and $Y_{n}$ are symmetric matrices. On the other hand, if $n$ is
even then $\Lambda_{n}$ is not symmetric. In fact, in that case the
equivalence class of $\Lambda_{n}$ contains no symmetric patterns. This
follows by comparing the row sums and column sums of $\Lambda_{n}$. For $n=6$,
we have verified that the maximal zero-patterns form two equivalence classes.

Now let $n=2m$ be even. Denote by $\Lambda_{n}^{\#}$ the $n\times n$ symmetric
zero-pattern in the following infinite sequence:
\[%
\begin{tabular}
[c]{lll}%
$\Lambda_{4}^{\#}=\left[
\begin{tabular}
[c]{rrrr}%
$1$ & $1$ & $1$ & $1$\\
$1$ & $1$ & $1$ & $1$\\\cline{4-4}%
$1$ & $1$ & $1$ & \multicolumn{1}{|r}{$\mathbf{0}$}\\\cline{3-3}%
$1$ & $1$ & \multicolumn{1}{|r}{$\mathbf{0}$} & $\mathbf{0}$%
\end{tabular}
\ \ \ \ \ \ \right]  ,$ &  & $\Lambda_{6}^{\#}=\left[
\begin{tabular}
[c]{rrrrrr}%
$\mathbf{0}$ & $\mathbf{0}$ & $\mathbf{0}$ & $\mathbf{0}$ &
\multicolumn{1}{|r}{$1$} & $1$\\\cline{2-4}%
$\mathbf{0}$ & \multicolumn{1}{|r}{$1$} & $1$ & $1$ & $1$ & $1$\\
$\mathbf{0}$ & \multicolumn{1}{|r}{$1$} & $1$ & $1$ & $1$ & $1$\\\cline{5-6}%
$\mathbf{0}$ & \multicolumn{1}{|r}{$1$} & $1$ & $1$ &
\multicolumn{1}{|r}{$\mathbf{0}$} & $\mathbf{0}$\\\cline{1-1}\cline{4-4}%
$1$ & $1$ & $1$ & \multicolumn{1}{|r}{$\mathbf{0}$} & $\mathbf{0}$ &
$\mathbf{0}$\\
$1$ & $1$ & $1$ & \multicolumn{1}{|r}{$\mathbf{0}$} & $\mathbf{0}$ &
$\mathbf{0}$%
\end{tabular}
\ \ \ \ \right]  ,$%
\end{tabular}
\ \ \ \
\]%
\[%
\begin{tabular}
[c]{l}%
$\Lambda_{8}^{\#}=\left[
\begin{tabular}
[c]{rrrrrrrr}%
$\mathbf{0}$ & $\mathbf{0}$ & $\mathbf{0}$ & $\mathbf{0}$ & $\mathbf{0}$ &
\multicolumn{1}{|r}{$1$} & $1$ & $1$\\
$\mathbf{0}$ & $\mathbf{0}$ & $\mathbf{0}$ & $\mathbf{0}$ & $\mathbf{0}$ &
\multicolumn{1}{|r}{$1$} & $1$ & $1$\\\cline{3-5}\cline{8-8}%
$\mathbf{0}$ & $\mathbf{0}$ & \multicolumn{1}{|r}{$1$} & $1$ & $1$ & $1$ & $1$
& \multicolumn{1}{|r}{$\mathbf{0}$}\\
$\mathbf{0}$ & $\mathbf{0}$ & \multicolumn{1}{|r}{$1$} & $1$ & $1$ & $1$ & $1$
& \multicolumn{1}{|r}{$\mathbf{0}$}\\\cline{6-7}%
$\mathbf{0}$ & $\mathbf{0}$ & \multicolumn{1}{|r}{$1$} & $1$ & $1$ &
\multicolumn{1}{|r}{$\mathbf{0}$} & $\mathbf{0}$ & $\mathbf{0}$\\\cline{1-2}%
\cline{5-5}%
$1$ & $1$ & $1$ & $1$ & \multicolumn{1}{|r}{$\mathbf{0}$} & $\mathbf{0}$ &
$\mathbf{0}$ & $\mathbf{0}$\\
$1$ & $1$ & $1$ & $1$ & \multicolumn{1}{|r}{$\mathbf{0}$} & $\mathbf{0}$ &
$\mathbf{0}$ & $\mathbf{0}$\\\cline{3-4}%
$1$ & $1$ & \multicolumn{1}{|r}{$\mathbf{0}$} & $\mathbf{0}$ & $\mathbf{0}$ &
$\mathbf{0}$ & $\mathbf{0}$ & $\mathbf{0}$%
\end{tabular}
\ \ \ \ \right]  ,$%
\end{tabular}
\ \ \ \
\]%
\[%
\begin{tabular}
[c]{l}%
$\Lambda_{10}^{\#}=\left[
\begin{tabular}
[c]{rrrrrrrrrr}%
$\mathbf{0}$ & $\mathbf{0}$ & $\mathbf{0}$ & $\mathbf{0}$ & $\mathbf{0}$ &
$\mathbf{0}$ & $\mathbf{0}$ & $\mathbf{0}$ & \multicolumn{1}{|r}{$1$} &
$1$\\\cline{7-8}%
$\mathbf{0}$ & $\mathbf{0}$ & $\mathbf{0}$ & $\mathbf{0}$ & $\mathbf{0}$ &
$\mathbf{0}$ & \multicolumn{1}{|r}{$1$} & $1$ & $1$ & $1$\\
$\mathbf{0}$ & $\mathbf{0}$ & $\mathbf{0}$ & $\mathbf{0}$ & $\mathbf{0}$ &
$\mathbf{0}$ & \multicolumn{1}{|r}{$1$} & $1$ & $1$ & $1$\\\cline{4-6}%
\cline{9-10}%
$\mathbf{0}$ & $\mathbf{0}$ & $\mathbf{0}$ & \multicolumn{1}{|r}{$1$} & $1$ &
$1$ & $1$ & $1$ & \multicolumn{1}{|r}{$\mathbf{0}$} & $\mathbf{0}$\\
$\mathbf{0}$ & $\mathbf{0}$ & $\mathbf{0}$ & \multicolumn{1}{|r}{$1$} & $1$ &
$1$ & $1$ & $1$ & \multicolumn{1}{|r}{$\mathbf{0}$} & $\mathbf{0}%
$\\\cline{7-8}%
$\mathbf{0}$ & $\mathbf{0}$ & $\mathbf{0}$ & \multicolumn{1}{|r}{$1$} & $1$ &
$1$ & \multicolumn{1}{|r}{$\mathbf{0}$} & $\mathbf{0}$ & $\mathbf{0}$ &
$\mathbf{0}$\\\cline{2-3}\cline{6-6}%
$\mathbf{0}$ & \multicolumn{1}{|r}{$1$} & $1$ & $1$ & $1$ &
\multicolumn{1}{|r}{$\mathbf{0}$} & $\mathbf{0}$ & $\mathbf{0}$ & $\mathbf{0}$
& $\mathbf{0}$\\
$\mathbf{0}$ & \multicolumn{1}{|r}{$1$} & $1$ & $1$ & $1$ &
\multicolumn{1}{|r}{$\mathbf{0}$} & $\mathbf{0}$ & $\mathbf{0}$ & $\mathbf{0}$
& $\mathbf{0}$\\\cline{1-1}\cline{4-5}%
$1$ & $1$ & $1$ & \multicolumn{1}{|r}{$\mathbf{0}$} & $\mathbf{0}$ &
$\mathbf{0}$ & $\mathbf{0}$ & $\mathbf{0}$ & $\mathbf{0}$ & $\mathbf{0}$\\
$1$ & $1$ & $1$ & \multicolumn{1}{|r}{$\mathbf{0}$} & $\mathbf{0}$ &
$\mathbf{0}$ & $\mathbf{0}$ & $\mathbf{0}$ & $\mathbf{0}$ & $\mathbf{0}$%
\end{tabular}
\ \ \ \right]  ,$%
\end{tabular}
\ldots
\]
Note that $\Lambda_{n}^{\#}$ has exactly $4n-3$ ones.

\begin{theorem}
For odd (resp. even) $n>2$ there exist symmetric rational orthogonal matrices
with zero-pattern $\Lambda_{n}$ (resp. $\Lambda_{n}^{\#}$).
\end{theorem}

\begin{proof}
We have already taken care of the odd case. In the even case, we shall
construct the required matrices $Z_{n}$ for $n=4,6$ and $8$ only. It will be
obvious how to proceed for bigger values of $n$. For $(x_{k},y_{k})$, $k\geq
0$, we can choose any rational point on the unit circle $x^{2}+y^{2}=1$ such
that $x_{k}y_{k}\neq0$. For $n=4,6$ we take the matrices in the forms
\[%
\begin{tabular}
[c]{l}%
$Z_{4}=\left[
\begin{array}
[c]{rrrr}%
x_{0}x_{1}^{2} & x_{0}x_{1}y_{1} & y_{0}x_{1} & y_{1}\\
x_{0}x_{1}y_{1} & x_{0}y_{1}^{2} & y_{0}y_{1} & -x_{1}\\
y_{0}x_{1} & y_{0}y_{1} & -x_{0} & \mathbf{0}\\
y_{1} & -x_{1} & \mathbf{0} & \mathbf{0}%
\end{array}
\right]  ,$%
\end{tabular}
\]%
\[%
\begin{tabular}
[c]{l}%
$Z_{6}=\left[
\begin{array}
[c]{rrrrrr}%
\mathbf{0} & \mathbf{0} & \mathbf{0} & \mathbf{0} & y_{2} & x_{2}\\
\mathbf{0} & x_{0}x_{1}^{2} & x_{0}x_{1}y_{1} & y_{0}x_{1} & y_{1}x_{2} &
-y_{1}y_{2}\\
\mathbf{0} & x_{0}x_{1}y_{1} & x_{0}y_{1}^{2} & y_{0}y_{1} & -x_{1}x_{2} &
x_{1}y_{2}\\
\mathbf{0} & y_{0}x_{1} & y_{0}y_{1} & -x_{0} & \mathbf{0} & \mathbf{0}\\
y_{2} & y_{1}x_{2} & -x_{1}x_{2} & \mathbf{0} & \mathbf{0} & \mathbf{0}\\
x_{2} & -y_{1}y_{2} & x_{1}y_{2} & \mathbf{0} & \mathbf{0} & \mathbf{0}%
\end{array}
\right]  .$%
\end{tabular}
\]
For $n=8$ we take
\[
Z_{8}=\left[
\begin{array}
[c]{rrrrrrrr}%
\mathbf{0} & \mathbf{0} & \mathbf{0} & \mathbf{0} & \mathbf{0} & y_{2}x_{3} &
x_{2}x_{3} & y_{3}\\
\mathbf{0} & \mathbf{0} & \mathbf{0} & \mathbf{0} & \mathbf{0} & y_{2}y_{3} &
x_{2}y_{3} & -x_{3}\\
\mathbf{0} & \mathbf{0} & x_{0}x_{1}^{2} & x_{0}x_{1}y_{1} & y_{0}x_{1} &
y_{1}x_{2} & -y_{1}y_{2} & \mathbf{0}\\
\mathbf{0} & \mathbf{0} & x_{0}x_{1}y_{1} & x_{0}y_{1}^{2} & y_{0}y_{1} &
-x_{1}x_{2} & x_{1}y_{2} & \mathbf{0}\\
\mathbf{0} & \mathbf{0} & y_{0}x_{1} & y_{0}y_{1} & -x_{0} & \mathbf{0} &
\mathbf{0} & \mathbf{0}\\
y_{2}x_{3} & y_{2}y_{3} & y_{1}x_{2} & -x_{1}x_{2} & \mathbf{0} & \mathbf{0} &
\mathbf{0} & \mathbf{0}\\
x_{2}x_{3} & x_{2}y_{3} & -y_{1}y_{2} & x_{1}y_{2} & \mathbf{0} & \mathbf{0} &
\mathbf{0} & \mathbf{0}\\
y_{3} & -x_{3} & \mathbf{0} & \mathbf{0} & \mathbf{0} & \mathbf{0} &
\mathbf{0} & \mathbf{0}%
\end{array}
\right]  .
\]
In each of these cases $Z_{n}$ is orthogonal and symmetric, and so involutory
matrix, with zero-pattern $\Lambda_{n}^{\#}$. In general, the construction can
be best understood by considering the infinite matrix%
%TCIMACRO{\TeXButton{\begingroup}{\begingroup}}%
%BeginExpansion
\begingroup
%EndExpansion%
%TCIMACRO{\TeXButton{\scalefont{x}}{\scalefont{0.7}}}%
%BeginExpansion
\scalefont{0.7}%
%EndExpansion%
\[
Z=\left[
\begin{tabular}
[c]{cccccccccccc}
&  &  &  &  &  &  &  &  &  &  & $\cdots$\\
&  &  &  &  &  &  &  &  & $y_{4}y_{5}$ & $x_{4}y_{5}$ & \\\cline{3-10}
&  & \multicolumn{1}{|c}{} &  &  &  &  & $y_{2}x_{3}$ & $x_{2}x_{3}$ &
$y_{3}x_{4}$ & \multicolumn{1}{|c}{$-y_{3}y_{4}$} & \\\cline{4-9}
&  & \multicolumn{1}{|c}{} & \multicolumn{1}{|c}{} &  &  &  & $y_{2}x_{3}$ &
$x_{2}y_{3}$ & \multicolumn{1}{|c}{$-x_{3}x_{4}$} & \multicolumn{1}{|c}{$x_{3}%
y_{4}$} & \\\cline{5-8}
&  & \multicolumn{1}{|c}{} & \multicolumn{1}{|c}{} &
\multicolumn{1}{|c}{$x_{0}x_{1}^{2}$} & $x_{0}x_{1}y_{1}$ & $y_{0}x_{1}$ &
$y_{1}x_{2}$ & \multicolumn{1}{|c}{$-y_{1}y_{2}$} & \multicolumn{1}{|c}{} &
\multicolumn{1}{|c}{} & \\
&  & \multicolumn{1}{|c}{} & \multicolumn{1}{|c}{} &
\multicolumn{1}{|c}{$x_{0}x_{1}y_{1}$} & $x_{0}y_{1}^{2}$ & $y_{0}y_{1}$ &
$-x_{1}x_{2}$ & \multicolumn{1}{|c}{$x_{1}y_{2}$} & \multicolumn{1}{|c}{} &
\multicolumn{1}{|c}{} & \\
&  & \multicolumn{1}{|c}{} & \multicolumn{1}{|c}{} &
\multicolumn{1}{|c}{$y_{0}x_{1}$} & $y_{0}y_{1}$ & $-x_{0}$ &  &
\multicolumn{1}{|c}{} & \multicolumn{1}{|c}{} & \multicolumn{1}{|c}{} & \\
&  & \multicolumn{1}{|c}{$y_{2}x_{3}$} & \multicolumn{1}{|c}{$y_{2}y_{3}$} &
\multicolumn{1}{|c}{$y_{1}x_{2}$} & $-x_{1}x_{2}$ &  &  &
\multicolumn{1}{|c}{} & \multicolumn{1}{|c}{} & \multicolumn{1}{|c}{} &
\\\cline{5-8}
&  & \multicolumn{1}{|c}{$x_{2}x_{3}$} & \multicolumn{1}{|c}{$x_{2}y_{3}$} &
$-y_{1}y_{2}$ & $x_{1}y_{2}$ &  &  &  & \multicolumn{1}{|c}{} &
\multicolumn{1}{|c}{} & \\\cline{4-9}
& $y_{4}y_{5}$ & \multicolumn{1}{|c}{$y_{3}y_{4}$} & $-x_{3}x_{4}$ &  &  &  &
&  &  & \multicolumn{1}{|c}{} & \\\cline{3-10}
& $x_{4}y_{5}$ & $-y_{3}y_{4}$ & $x_{3}y_{4}$ &  &  &  &  &  &  &  & \\
$\cdots$ &  &  &  &  &  &  &  &  &  &  &
\end{tabular}
\ \ \ \right]  .
\]%
%TCIMACRO{\TeXButton{\endgroup}{\endgroup}}%
%BeginExpansion
\endgroup
%EndExpansion

\end{proof}

\subsection{Symmetric rational orthogonal matrices from hypercubes}

The \emph{Hamming distance} between two words $v,w\in\left\{  0,1\right\}
^{n}$, denoted by $d\left(  v,w\right)  $, is the number of coordinates in
which the words differ: $d\left(  v,w\right)  :=\sum_{i=1}^{n}\left\vert
v_{i}-w_{i}\right\vert $. The $n$\emph{-dimensional hypercube}, denoted by
$Q_{n}$, is the graph defined as follows: $V(Q_{n})=\{0,1\}^{n}$; $\{v,w\}\in
E(Q_{n})$ if and only $d(v,w)=1$. The adjacency matrix of $Q_{n}$ can be
constructed recursively:
\[%
\begin{tabular}
[c]{lll}%
$M(Q_{1})=\left(
\begin{array}
[c]{cc}%
0 & 1\\
1 & 0
\end{array}
\right)  ,$ & $M(Q_{n})=\left(
\begin{array}
[c]{cc}%
M(Q_{n-1}) & I\\
I & M(Q_{n-1})
\end{array}
\right)  ,$ & for $n\geq2.$%
\end{tabular}
\ \
\]
It is simple to verify that the following matrices are real orthogonal:
\[%
\begin{tabular}
[c]{lll}%
$M_{1}=\left[
\begin{array}
[c]{cc}%
0 & -1\\
1 & 0
\end{array}
\right]  ,$ & $M_{n}=\frac{1}{\sqrt{n}}\left[
\begin{array}
[c]{cc}%
M_{n-1} & -I\\
I & M_{n-1}^{-1}%
\end{array}
\right]  ,$ & for $n\geq2$.
\end{tabular}
\ \
\]
Clearly, $\underline{M_{n}}=M(Q_{n})$.

\begin{proposition}
For each $n\geq2$, there exist many infinitely symmetric matrices $Y_{n}\in
O_{2^{n}}(\mathbb{Q})$ such that $\underline{Y_{n}}=M(Q_{n})$.
\end{proposition}

\begin{proof}
Let $x_{1},x_{2},...$ be indeterminates. We define the matrices $X_{n}$
recursively by%
\[%
\begin{tabular}
[c]{lll}%
$X_{1}=\left[
\begin{array}
[c]{cc}%
0 & x_{1}\\
x_{1} & 0
\end{array}
\right]  ,$ & $X_{n}=\left[
\begin{array}
[c]{cc}%
X_{n-1} & x_{n}I\\
x_{n}I & -X_{n-1}%
\end{array}
\right]  ,$ & for $n\geq2$.
\end{tabular}
\ \
\]
Note that $X_{n}^{T}=X_{n}$ and one can easily verify by induction that
\[
X_{n}^{2}=\left(  x_{1}^{2}+x_{2}^{2}+\cdots+x_{n}^{2}\right)  I_{2^{n}}.
\]
For a given $n\geq1$, we choose nonzero rational numbers $\alpha_{1}%
,\alpha_{2},...,\alpha_{n}$ such that $\alpha_{1}^{2}+\alpha_{2}^{2}%
+\cdots+\alpha_{n}^{2}=1$. Then if we set $x_{k}=\alpha_{k}$ ($k=1,2,...,n$)
in $X_{n}$, we obtain a symmetric rational orthogonal matrix $Y_{n}$ with
$\underline{Y_{n}}=M(Q_{n})$.
\end{proof}

\subsection{Hessenberg rational orthogonal matrices}

Let $H_{n}$ be the lower triangular $n\times n$ Hessenberg zero-pattern:%
\[
H_{n}=\left[
\begin{array}
[c]{cccccc}%
1 & 1 & 1 & 1 & \cdots & 1\\
1 & 1 & 1 & 1 &  & 1\\
0 & 1 & 1 & 1 &  & \vdots\\
& 0 & 1 & 1 &  & \\
&  & \ddots & \ddots &  & \\
\mathbf{0} &  &  & 0 &  & 1
\end{array}
\right]  .
\]
We consider here the corresponding symmetric zero-pattern $S_{n}H_{n}$, where
\[
S_{n}=\left[
\begin{array}
[c]{cccc}%
\mathbf{0} &  &  & 1\\
&  & 1 & \\
& \cdots &  & \\
1 &  &  & \mathbf{0}%
\end{array}
\right]
\]
is the antidiagonal permutation matrix.

\begin{proposition}
There exist infinitely many $Y_{n}=Y_{n}^{T}\in O_{n}(\mathbb{Q})$ with
$\underline{Y_{n}}=S_{n}H_{n}$.
\end{proposition}

\begin{proof}
Let $a,b\in\mathbb{Q}$ be nonzero and such that $a^{2}+b^{2}=1$. We define
recursively the matrices $X_{n}\in O_{n}(\mathbb{Q})$, $n\geq2$, by%
\[
X_{2}=\left[
\begin{array}
[c]{cc}%
a & -b\\
b & a
\end{array}
\right]  ;
\]%
\[
X_{n}=\left[
\begin{array}
[c]{cc}%
X_{n-1} & \mathbf{0}\\
\mathbf{0} & 1
\end{array}
\right]  \left[
\begin{array}
[c]{cc}%
I_{n-2} & \mathbf{0}\\
\mathbf{0} & X_{2}%
\end{array}
\right]  ,
\]
for $n\geq3$. The matrices $Y_{n}=S_{n}H_{n}$ are symmetric and satisfy
$\underline{Y_{n}}=S_{n}H_{n}$. Indeed, we have
\[%
\begin{tabular}
[c]{lll}%
$Y_{2}=\left[
\begin{array}
[c]{rr}%
b & a\\
a & -b
\end{array}
\right]  ,$ &  & $Y_{3}=\left[
\begin{array}
[c]{rrr}%
\mathbf{0} & b & a\\
b & a^{2} & -ab\\
a & -ab & b^{2}%
\end{array}
\right]  $%
\end{tabular}
\ \
\]%
\[%
\begin{tabular}
[c]{lll}%
$Y_{4}=\left[
\begin{array}
[c]{rrrr}%
\mathbf{0} & \mathbf{0} & b & a\\
\mathbf{0} & b & a^{2} & -ab\\
b & a^{2} & -a^{2}b & ab^{2}\\
a & -ab & ab^{2} & -b^{3}%
\end{array}
\right]  ,$ & and & $Y_{5}=\left[
\begin{array}
[c]{rrrrr}%
\mathbf{0} & \mathbf{0} & \mathbf{0} & b & a\\
\mathbf{0} & \mathbf{0} & b & a^{2} & -ab\\
\mathbf{0} & b & a^{2} & -a^{2}b & ab^{2}\\
b & a^{2} & -a^{2}b & a^{2}b^{2} & -ab^{3}\\
a & -ab & ab^{2} & -ab^{3} & b^{4}%
\end{array}
\right]  .$%
\end{tabular}
\ \
\]
By induction on $n$ one can prove that%
%TCIMACRO{\TeXButton{\begingroup}{\begingroup}}%
%BeginExpansion
\begingroup
%EndExpansion%
%TCIMACRO{\TeXButton{\scalefont{x}}{\scalefont{0.8}}}%
%BeginExpansion
\scalefont{0.8}%
%EndExpansion%
\[
Y_{n}=\left[
\begin{array}
[c]{rrrrrrrrrr}
&  &  &  &  & \mathbf{0} & \mathbf{0} & \mathbf{0} & b & a\\
&  &  &  &  & \mathbf{0} & \mathbf{0} & b & a^{2} & -ab\\
& \mathbf{0} &  &  &  & \mathbf{0} & b & a^{2} & -a^{2}b & ab^{2}\\
&  &  &  &  & b & a^{2} & -a^{2}b & a^{2}b^{2} & -ab^{3}\\
&  &  &  & \cdots &  &  &  &  & \\
\mathbf{0} & \mathbf{0} & \mathbf{0} & b &  & a^{2}\left(  -b\right)  ^{n-9} &
a^{2}\left(  -b\right)  ^{n-8} & a^{2}\left(  -b\right)  ^{n-7} & a^{2}\left(
-b\right)  ^{n-6} & a\left(  -b\right)  ^{n-5}\\
\mathbf{0} & \mathbf{0} & b & a^{2} &  & a^{2}\left(  -b\right)  ^{n-8} &
a^{2}\left(  -b\right)  ^{n-7} & a^{2}\left(  -b\right)  ^{n-6} & a^{2}\left(
-b\right)  ^{n-5} & a\left(  -b\right)  ^{n-4}\\
\mathbf{0} & b & a^{2} & -a^{2}b &  & a^{2}\left(  -b\right)  ^{n-7} &
a^{2}\left(  -b\right)  ^{n-6} & a^{2}\left(  -b\right)  ^{n-5} & a^{2}\left(
-b\right)  ^{n-4} & a\left(  -b\right)  ^{n-3}\\
b & a^{2} & -a^{2}b & a^{2}b^{2} &  & a^{2}\left(  -b\right)  ^{n-6} &
a^{2}\left(  -b\right)  ^{n-5} & a^{2}\left(  -b\right)  ^{n-4} & a^{2}\left(
-b\right)  ^{n-3} & a\left(  -b\right)  ^{n-2}\\
a & -ab & ab^{2} & -ab^{3} &  & a\left(  -b\right)  ^{n-5} & a\left(
-b\right)  ^{n-4} & a\left(  -b\right)  ^{n-3} & a\left(  -b\right)  ^{n-2} &
\left(  -b\right)  ^{n-1}%
\end{array}
\right]  .
\]%
%TCIMACRO{\TeXButton{\endgroup}{\endgroup}}%
%BeginExpansion
\endgroup
%EndExpansion

\end{proof}

\bigskip

Note that the trace of $Y_{n}$ is zero for $n$ even and one for $n$ odd.

\section{Open problems\label{sec5}}

In addition to the conjectures formulated in the paper, we state here some
further open problems.

The first problem is of purely combinatorial nature.

\begin{problem}
Let $X$ be an $n\times n$ zero-pattern and assume that $X^{T}$ is equivalent
to $X$. Is it true that $X$ is equivalent to a symmetric pattern?
\end{problem}

We have verified that the answer to the above problem is affirmative for
$n\leq5$.

\begin{problem}
Are there symmetric rational orthogonal matrices with the following
zero-patterns:
\[%
\begin{tabular}
[c]{lll}%
$\left[
\begin{array}
[c]{rrrrr}%
1 & \mathbf{0} & 1 & 1 & 1\\
\mathbf{0} & 1 & 1 & 1 & 1\\
1 & 1 & \mathbf{0} & 1 & 1\\
1 & 1 & 1 & \mathbf{0} & 1\\
1 & 1 & 1 & 1 & \mathbf{0}%
\end{array}
\right]  $ & , & $\left[
\begin{array}
[c]{rrrrr}%
1 & 1 & 1 & 1 & \mathbf{0}\\
1 & 1 & 1 & \mathbf{0} & 1\\
1 & 1 & \mathbf{0} & 1 & 1\\
1 & \mathbf{0} & 1 & \mathbf{0} & 1\\
\mathbf{0} & 1 & 1 & 1 & 1
\end{array}
\right]  $%
\end{tabular}
\ ?
\]

\end{problem}

As mentioned in Section \ref{sec3}, in spite of much effort we were not able
to construct such matrices. Below we give some examples of matrices with the
same zero-pattern \emph{as close as possible }to be rational. For the first
zero-patten we give two examples. The first one minimizes the denominator:%
%TCIMACRO{\TeXButton{\begingroup}{\begingroup}}%
%BeginExpansion
\begingroup
%EndExpansion%
%TCIMACRO{\TeXButton{\scalefont{x}}{\scalefont{1}}}%
%BeginExpansion
\scalefont{1}%
%EndExpansion%
\[
\frac{1}{4}\left[
\begin{array}
[c]{rrrrr}%
2 & \mathbf{0} & 2 & \sqrt{3}-1 & -\sqrt{3}-1\\
\mathbf{0} & 2 & 2 & -\sqrt{3}-1 & \sqrt{3}-1\\
2 & 2 & \mathbf{0} & 2 & 2\\
\sqrt{3}-1 & -\sqrt{3}-1 & 2 & \mathbf{0} & 2\\
-\sqrt{3}-1 & \sqrt{3}-1 & 2 & 2 & \mathbf{0}%
\end{array}
\right]  _{5,2}^{1},
\]%
\[
\frac{1}{21}\left[
\begin{array}
[c]{rrrrr}%
16 & \mathbf{0} & 10 & 4\sqrt{5} & \sqrt{5}\\
\mathbf{0} & 5 & 4 & -4\sqrt{5} & 8\sqrt{5}\\
10 & 4 & \mathbf{0} & -7\sqrt{5} & -4\sqrt{5}\\
4\sqrt{5} & -4\sqrt{5} & -7\sqrt{5} & \mathbf{0} & 6\\
\sqrt{5} & 8\sqrt{5} & -4\sqrt{5} & 6 & \mathbf{0}%
\end{array}
\right]  _{5,2}^{1}.
\]
We give four examples for the next zero-pattern. The first example minimizes
the square root, the second has the smallest denominator, the third minimizes
the number of square roots, and the last one contains the smallest prime
number under square root:%
\[
\frac{1}{245}\left[
\begin{array}
[c]{rrrrr}%
225 & 40 & 18\sqrt{15} & -14\sqrt{15} & \mathbf{0}\\
40 & 45 & -40\sqrt{15} & \mathbf{0} & 180\\
18\sqrt{15} & -40\sqrt{15} & \mathbf{0} & 175 & 6\sqrt{15}\\
-14\sqrt{15} & \mathbf{0} & 175 & \mathbf{0} & 42\sqrt{15}\\
\mathbf{0} & 180 & 6\sqrt{15} & 42\sqrt{15} & -25
\end{array}
\right]  _{12}^{1},
\]%
\[
\frac{1}{108}\left[
\begin{array}
[c]{rrrrr}%
98 & 28 & 7\sqrt{22} & -3\sqrt{22} & \mathbf{0}\\
28 & 12 & -20\sqrt{22} & \mathbf{0} & 44\\
7\sqrt{22} & -20\sqrt{22} & \mathbf{0} & 42 & \sqrt{22}\\
-3\sqrt{22} & \mathbf{0} & 42 & \mathbf{0} & 21\sqrt{22}\\
\mathbf{0} & 44 & \sqrt{22} & 21\sqrt{22} & -2
\end{array}
\right]  _{12}^{1},
\]%
\[
\frac{1}{2527}\left[
\begin{array}
[c]{rrrrr}%
2457 & 12\sqrt{195} & -420 & 380 & \mathbf{0}\\
12\sqrt{195} & 343 & 80\sqrt{195} & \mathbf{0} & -160\sqrt{195}\\
-420 & 80\sqrt{195} & \mathbf{0} & 2223 & 140\\
380 & \mathbf{0} & 2223 & \mathbf{0} & 1140\\
\mathbf{0} & -160\sqrt{195} & 140 & 1140 & -273
\end{array}
\right]  _{12}^{1},
\]%
\[
\frac{1}{418241}\left[
\begin{array}
[c]{rrrrr}%
403200 & 5460\sqrt{17} & -79040 & 74841 & \mathbf{0}\\
5460\sqrt{17} & 110864 & 35511\sqrt{17} & \mathbf{0} & -90972\sqrt{17}\\
-79040 & 35511\sqrt{17} & \mathbf{0} & 381780 & 38532\\
74841 & \mathbf{0} & 381780 & \mathbf{0} & 153520\\
\mathbf{0} & -90972\sqrt{17} & 153520 & 38532 & -95823
\end{array}
\right]  _{12}^{1}.
\]%
%TCIMACRO{\TeXButton{\endgroup}{\endgroup}}%
%BeginExpansion
\endgroup
%EndExpansion

In connection with Proposition \ref{conf}, we raise the following special case
of Conjecture \ref{con3} as a separate interesting problem.

\begin{problem}
For even $n$, show that there exists $X=X^{T}\in O_{n}(\mathbb{Q})$ with
$\underline{X}=J_{n}-I_{n}$. For odd $n\geq3$, show that there exists
$X=X^{T}\in O_{n}(\mathbb{Q})$ with $\underline{X}=\Delta_{n,n-2}$ and
Tr$\left(  X\right)  =1$.
\end{problem}

Many of the matrices in Section \ref{sec2} and Section \ref{sec3} have been
constructed with the help of a computer. It is natural to raise the following problem:

\begin{problem}
Determine the computational complexity of the following decision problem:

\begin{itemize}
\item \noindent\textbf{Given:} A square $(0,1)$-matrix $M$ of size $n$.

\item \noindent\textbf{Task:} Determine if $M$ is the zero-pattern of a real
orthogonal matrix.
\end{itemize}

The size of $M$ gives the length of the input.
\end{problem}

\end{document}